\documentclass[a4paper]{article}
\usepackage[margin=30mm]{geometry}
\usepackage{amsmath}
\usepackage{amssymb}
\usepackage{amsthm}
\usepackage{titlesec}
\usepackage{titlefoot}
\usepackage{bm}
\titleformat*{\section}{\normalsize\bfseries}
\titleformat*{\subsection}{\normalsize\bfseries}
\allowdisplaybreaks
\def\R{\mathbb{R}}
\def\e{{\varepsilon}}        
\def\wh{\widehat}
\def\p{\partial}
\def\d{\displaystyle}
\newtheorem{thm}{Theorem}[section]
\newtheorem{lem}[thm]{Lemma}
\newtheorem{cor}[thm]{Corollary}
\newtheorem{prop}[thm]{Proposition}
\newtheorem{rem}[thm]{Remark}

\makeatletter
\@addtoreset{equation}{section}
 
  \makeatother
 
\begin{document}

\title{
\vspace{-1cm}
\large{\bf LARGE TIME BEHAVIOR OF SOLUTIONS TO THE CAUCHY PROBLEM FOR THE BBM--BURGERS EQUATION}}
\author{Ikki Fukuda and Masahiro Ikeda
}
\date{}
\maketitle


\vspace{-0.75cm}
\begin{abstract}
We consider the large time behavior of the solutions to the Cauchy problem for the BBM--Burgers equation. We prove that the solution to this problem goes to the self-similar solution to the Burgers equation called the nonlinear diffusion wave. Moreover, we construct the appropriate second asymptotic profiles of the solutions depending on the initial data. Based on that discussion, we investigate the effect of the initial data on the large time behavior of the solution, and derive the optimal asymptotic rate to the nonlinear diffusion wave. Especially, the important point of this study is that the second asymptotic profiles of the solutions with slowly decaying data, whose case has not been studied, are obtained.
\end{abstract}


\section{Introduction}

We consider the Cauchy problem for the following BBM--Burgers equation:
\begin{align}
\label{BBMB}
\begin{split}
u_{t}-u_{xxt}-u_{xx}+\gamma u_{xxx}+\beta uu_{x}&=0, \ \ x\in \R, \ t>0,\\
u(x, 0)&=u_{0}(x) , \ \ x\in \R, 
\end{split}
\end{align}
where $u=u(x, t)$ is a real-valued unknown function, $u_{0}(x)$ is a given initial data, $\beta\neq 0$ and $\gamma \in \R$. The subscripts $t$ and $x$ denote the partial derivatives with respect to $t$ and $x$, respectively. In the present paper, we assume that $u_{0}(x)$ is continuous and satisfies the following condition: 
\begin{equation}\label{data}
\exists \alpha>1, \ \ \exists C>0 \ \ s.t. \ \ |u_{0}(x)|\le C(1+|x|)^{-\alpha}, \ \ x\in \R.  
\end{equation}
The purpose of our study is to obtain the large time asymptotic profile of the solutions to \eqref{BBMB}. In particular, we investigate how the change of the decay rate $\alpha$ of the initial data affects its asymptotic behavior, and construct the appropriate second asymptotic profiles of the solution. In addition, by using the obtained second asymptotic profiles, we would like to derive the optimal convergence rates to the first asymptotic profile of the solution. 

First of all, we recall known results about the Cauchy problem \eqref{BBMB}. For the global existence and appropriate time decay estimates for the solutions to \eqref{BBMB} in Sobolev spaces, we can refer to \cite{ABS89, L95, M99}. The large time asymptotic behavior of the solutions to \eqref{BBMB} was first treated by Naumkin \cite{N93}. Moreover, Hayashi--Kaikina--Naumkin~\cite{HKN07} generalized this result in \cite{N93}. More precisely, they showed that if $\beta=1$ and the initial data $u_{0}\in H^{1}(\R)\cap W^{1, 1}(\R)$, then there exists a unique global solution $u\in C^{0}([0, \infty); H^{1}(\R)\cap W^{1, 1}(\R))$ satisfying the following asymptotics: 
\begin{equation}\label{hayashi-1st-1}
\lim_{t\to \infty}t^{\frac{1}{2}}\left\|u(\cdot, t)-t^{-\frac{1}{2}}f_{M}\left((\cdot)t^{-\frac{1}{2}}\right)\right\|_{L^{\infty}}
=0,
\end{equation}
where $f_{M}(x)$ and $M$ are defined by 
\begin{equation}\label{self-similar}
f_{M}(x):=-2\p_{x}\log\left(\cosh\frac{M}{4}-\sinh\left(\frac{M}{4}\right)\frac{2}{\sqrt{\pi}}\int_{0}^{x/2}e^{-y^{2}}dy\right), \quad
M:=\int_{\R}u_{0}(x)dx\neq0. 
\end{equation}
Moreover, if additionally the initial data $u_{0}\in L^{1}_{1}(\R)$, then the following asymptotics is true: 
\begin{equation}\label{hayashi-1st-2}
\left\|u(\cdot, t)-t^{-\frac{1}{2}}f_{M}\left((\cdot)t^{-\frac{1}{2}}\right)\right\|_{L^{\infty}}
\le Ct^{-\frac{1}{2}-\delta}
\end{equation}
for sufficiently large $t>0$, where $\delta\in (0, 1/2)$ and the weighted Lebesgue space $L_{1}^{1}(\R)$ is defined by 
\begin{equation*}
L^{1}_{1}(\R):=\left\{ f \in L^{1}(\R); \ \|f\|_{L^{1}_{1}}:=\int_{\R}|f(x)|(1+|x|)dx<\infty \right\}.
\end{equation*}
\newpage
\noindent
The asymptotic function $t^{-\frac{1}{2}}f_{M}(xt^{-\frac{1}{2}})$ is the self-similar solution for the Burgers equation \cite{C51, H50}: 
\begin{equation}\label{Burgers-standard}
u_{t}+uu_{x}-u_{xx}=0, \ \ x\in \R, \ t>0. 
\end{equation}
We note that the above self-similar solution is often expressed by using the so-called nonlinear diffusion wave $\chi(x, t)$ defined below (see e.g.~\cite{F19-1, F19-2, K07, MN04}). Actually, if we set
\begin{equation}\label{diffusion-wave}
\chi(x, t):=\frac{1}{\sqrt{1+t}} \chi_{*} \left(\frac{x}{\sqrt{1+t}}\right), \quad x\in \R, \ \ t\ge0,
\end{equation}
where
\begin{equation}\label{diffusion-wave*}
\chi_{*}(x):=\frac{1}{\beta}\frac{(e^{\beta M/2}-1)e^{-x^{2}/4}}{\sqrt{\pi}+(e^{\beta M/2}-1)\int_{x/2}^{\infty}e^{-y^{2}}dy}, \quad M:=\int_{\R}u_{0}(x)dx\neq0, 
\end{equation}
then, we can easily see that $f_{M}(x)=\chi_{*}(x)$ i.e. $t^{-\frac{1}{2}}f_{M}(xt^{-\frac{1}{2}})=\chi(x, t-1)$ when $\beta=1$. Moreover $\chi(x, t)$ satisfies the following Burgers equation with a conservation law:  
\begin{equation}\label{Burgers}
\chi_{t}+\beta \chi \chi_{x}-\chi_{xx}=0, \quad \int_{\R}\chi(x, t)dx=M. 
\end{equation}

Furthermore, Hayashi--Kaikina--Naumkin~\cite{HKN07} also studied the optimality of the asymptotic rate \eqref{hayashi-1st-2} to the self-similar solution $t^{-\frac{1}{2}}f_{M}(xt^{-\frac{1}{2}})$ for the Burgers equation \eqref{Burgers-standard} by constructing the second asymptotic profile of the solutions to \eqref{BBMB}. More precisely, if $u_{0}\in H^{1}(\R)\cap W^{1, 1}(\R)\cap L^{1}_{1}(\R)$, then the following estimate is established: 
\begin{equation}\label{hayashi-2nd}
\left\|u(\cdot, t)-t^{-\frac{1}{2}}f_{M}\left((\cdot)t^{-\frac{1}{2}}\right)-t^{-1}\log t\,\widetilde{f}_{M}\left((\cdot)t^{-\frac{1}{2}}\right)\right\|_{L^{\infty}}\le Ct^{-1}
\end{equation}
for sufficiently large $t>0$, where $f_{M}(x)$ is defined by \eqref{self-similar} and  $\widetilde{f}_{M}(x)$ is defined by 
\begin{align}
\widetilde{f}_{M}(x)&:=-\frac{\gamma (f_{M}(x)-x)e^{-\frac{x^{2}}{4}}}{32\sqrt{\pi}H(x)}\int_{\R}H(y)(f_{M}(y))^{3}dy, \label{hayashi-second}\\
H(x)&:=\cosh\frac{M}{4}-\sinh\left(\frac{M}{4}\right)\frac{2}{\sqrt{\pi}}\int_{0}^{x/2}e^{-y^{2}}dy. \label{hayashi-H}
\end{align}
From \eqref{hayashi-2nd}, the triangle inequality and \eqref{hayashi-second}, one can obtain the optimal asymptotic relation: 
\begin{equation}\label{hayashi-optimal}
\left\|u(\cdot, t)-t^{-\frac{1}{2}}f_{M}\left((\cdot)t^{-\frac{1}{2}}\right)\right\|_{L^{\infty}}=\left(\left\|\widetilde{f}_{M}(\cdot)\right\|_{L^{\infty}}+o(1)\right)t^{-1}\log t \quad as \ \ t\to \infty.
\end{equation}
Therefore, if $\beta=1$, $\gamma \neq0$ and $M \neq0$, we can see that the solution $u(x, t)$ to the Cauchy problem \eqref{BBMB} converges to the self-similar solution $t^{-\frac{1}{2}}f_{M}(xt^{-\frac{1}{2}})$ at the optimal asymptotic rate of $t^{-1}\log t$. This means that we cannot take $\delta=1/2$ in \eqref{hayashi-1st-2} and the logarithmic term $\log t$ in \eqref{hayashi-optimal} is an essential factor. The similar estimates as \eqref{hayashi-1st-2}, \eqref{hayashi-2nd} and \eqref{hayashi-optimal} 
and other related results are also obtained for several dissipative type equations, such as the generalized BBM--Burgers equation \cite{K97-1, K99-1, K97-2, HKNS06}, the KdV--Burgers equation \cite{F19-1, HN06, KP05, K99-2, HKNS06}, the generalized Burgers equation \cite{K07, MN04} and also the damped wave equation with a nonlinear convection term \cite{KU17}.

In the above previous results, the initial data $u_{0}(x)$ is assumed to be decaying fast as $|x|\rightarrow\infty$, i.e. $u_{0}\in L^{1}_{1}(\R)$ to prove \eqref{hayashi-1st-2} and \eqref{hayashi-2nd}. To realize this condition, we need to assume $\alpha>2$ in \eqref{data}. However, for the problem \eqref{BBMB} in the case of $\alpha\le 2$ in \eqref{data}, the second asymptotic profiles, and the optimal asymptotic rate to the nonlinear diffusion wave are not known. For this reason, in the present paper, we would like to analyze \eqref{BBMB} in the case that the initial data decays more slowly, i.e. in the case of $1<\alpha \le 2$ in \eqref{data}. Especially, we construct the appropriate second asymptotic profiles depending on the initial data to the solution. Based on that discussion, we investigate the effect of the initial data on the large time behavior of the solution, and derive the optimal asymptotic rate to the nonlinear diffusion wave. To obtain the second asymptotic profiles, we develop the method used in \cite{F19-2, F20} to study the Cauchy problem of the KdV--Burgers equation \cite{F19-2} and of the damped wave equation with a nonlinear convection term \cite{F20}. 

\medskip
Now, let us state our first main result which generalizes the results given in \cite{HKN07} and treat the case of the initial data decays more slowly: 
\begin{thm}\label{main.thm-1st.AP}
Let $s\ge1$ be an integer. Assume that the initial data $u_{0}(x)$ satisfies the condition \eqref{data}, $u_{0}\in H^{s}(\R)$ and $\|u_{0}\|_{H^{s}}+\|u_{0}\|_{L^{1}}$ is sufficiently small. Then, \eqref{BBMB} has a unique global mild solution $u\in C^{0}([0, \infty); H^{s}(\R))$. Moreover, for any $\e>0$, the estimate
\begin{align}\label{main.thm-1st.AP-L2}
\left\|\p_{x}^{l}(u(\cdot, t)-\chi(\cdot, t))\right\|_{L^{2}}\le C\begin{cases}
(1+t)^{-\frac{\alpha}{2}+\frac{1}{4}-\frac{l}{2}}, &t\ge0, \ 1<\alpha<2,\\
(1+t)^{-\frac{3}{4}-\frac{l}{2}+\e}, &t\ge0, \ \alpha\ge2
\end{cases}
\end{align}
holds for any integer $l$ satisfying $0\le l\le s$. Furthermore, 
\begin{align}\label{main.thm-1st.AP-Linfinity}
\left\|\p_{x}^{l}(u(\cdot, t)-\chi(\cdot, t))\right\|_{L^{\infty}}\le C\begin{cases}
(1+t)^{-\frac{\alpha}{2}-\frac{l}{2}}, &t\ge0, \ 1<\alpha<2,\\
(1+t)^{-1-\frac{l}{2}+\e}, &t\ge0, \ \alpha\ge2
\end{cases}
\end{align}
holds for any integer $l$ satisfying $0\le l\le s-1$. Here, $\chi(x, t)$ is defined by \eqref{diffusion-wave}. 
\end{thm}

Next we state our second main result, in which we construct the second asymptotic profiles of the solution depending on the decay rate of the initial data:
\begin{thm}\label{main.thm-2nd.AP}
Let $s\ge2$ be an integer. Assume that the initial data $u_{0}(x)$ satisfies the condition \eqref{data}, $u_{0}\in H^{s}(\R)$ and $\|u_{0}\|_{H^{s}}+\|u_{0}\|_{L^{1}}$ is sufficiently small. In addition, we set
\begin{equation}\label{data-r0}
r_{0}(x):=\eta_{*}(x)^{-1}\int_{-\infty}^{x}\left(u_{0}(y)-\chi_{*}(y)\right)dy, \quad 
\eta_{*}(x):=\exp \left(\frac{\beta}{2}\int_{-\infty}^{x}\chi_{*}(y)dy\right),
\end{equation} 
where $\chi_{*}(x)$ is defined by \eqref{diffusion-wave*}, and suppose that there exist $\lim_{x\to \pm \infty}(1+|x|)^{\alpha-1}r_{0}(x)=:c_{\alpha}^{\pm}$. Then, the solution to \eqref{BBMB} satisfies
\begin{alignat}{3}
&\lim_{t\to \infty}(1+t)^{\frac{\alpha}{2}+\frac{l}{2}}\left\|\p_{x}^{l}(u(\cdot, t)-\chi(\cdot, t)-Z(\cdot, t))\right\|_{L^{\infty}}=0, &\quad &1<\alpha<2, \label{main.thm-2nd.AP-1}\\
&\lim_{t\to \infty}\frac{(1+t)^{1+\frac{l}{2}}}{\log(1+t)}\left\|\p_{x}^{l}(u(\cdot, t)-\chi(\cdot, t)-Z(\cdot, t)-V(\cdot, t))\right\|_{L^{\infty}}=0, &\quad &\alpha=2,  \label{main.thm-2nd.AP-2} \\
&\left\|\p_{x}^{l}(u(\cdot, t)-\chi(\cdot, t)-V(\cdot, t))\right\|_{L^{\infty}}\le C(1+t)^{-1-\frac{l}{2}}, \ \ t\ge1, &\quad &\alpha>2 \label{main.thm-2nd.AP-3}
\end{alignat}
for all integer $l$ satisfying $0\le l\le s-2$. Here, $Z(x, t)$ is defined by   
\begin{align}
Z(x, t)&:=\int_{\R}\frac{c_{\alpha}(y)\p_{x}(G(x-y, t)\eta(x, t))}{(1+|y|)^{\alpha-1}}dy, \quad  
c_{\alpha}(y):=\begin{cases}
c_{\alpha}^{+}, &y\ge0, \\
c_{\alpha}^{-}, &y<0,
\end{cases} \label{Second-AP-Z} \\
G(x, t)&:=\frac{1}{\sqrt{4\pi t}}e^{-\frac{x^{2}}{4t}}, \quad 
\eta(x, t):=\eta_{*}\left(\frac{x}{\sqrt{1+t}}\right)
=\exp \left(\frac{\beta}{2}\int_{-\infty}^{x}\chi(y, t)dy\right), \label{heat-eta}
\end{align}
and $V(x, t)$ is defined by 
\begin{align}
V(x, t)&:=-\kappa dV_{*}\left(\frac{x}{\sqrt{1+t}}\right)(1+t)^{-1}\log(1+t), \label{Second-AP-V} \\
V_{*}(x)&:=\frac{1}{\sqrt{4\pi}}\frac{d}{dx}\left(\eta_{*}(x)e^{-\frac{x^{2}}{4}}\right)
=\frac{1}{4\sqrt{\pi}}\left(\beta \chi_{*}(x)-x\right)\eta_{*}(x)e^{-\frac{x^{2}}{4}}, \label{V*}\\
d&:=\int_{\R}(\eta_{*}(y))^{-1}(\chi_{*}(y))^{3}dy, \quad \kappa:=\frac{\beta^{2}\gamma}{8}. \label{kappa}
\end{align}
\end{thm}

\begin{rem}{\rm
The similar results as \eqref{main.thm-2nd.AP-1} and \eqref{main.thm-2nd.AP-2} with $l=0$ are obtained by the first author, for the generalized KdV--Burgers equation \cite{F19-1} and the damped wave equation with a nonlinear convection term \cite{F20}. We emphasize that our results \eqref{main.thm-2nd.AP-1} and \eqref{main.thm-2nd.AP-2} are more general and include the estimates for derivatives of the solutions. In addition, our result \eqref{main.thm-2nd.AP-3} includes \eqref{hayashi-2nd}, which is given by Hayashi--Kaikina--Naumkin~\cite{HKN07}. Indeed, the second asymptotic profile $V(x, t)$ given by us and $t^{-1}\log t\,\widetilde{f}_{M}(xt^{-\frac{1}{2}})$ given in \cite{HKN07} are essentially the same. More precisely, by using \eqref{self-similar}, \eqref{diffusion-wave}, \eqref{diffusion-wave*}, \eqref{hayashi-second}, \eqref{hayashi-H},  \eqref{Second-AP-V} and \eqref{V*}, we can easily show that $\widetilde{f}_{M}(x)=-\kappa dV_{*}(x)$ i.e. $t^{-1}\log t\,\widetilde{f}_{M}(xt^{-\frac{1}{2}})=V(x, t-1)$ with $\beta=1$. Therefore, Theorem~\ref{main.thm-2nd.AP} is a generalization of their results. 
}\end{rem}

In view of the second asymptotic profiles, we are able to investigate the optimality of the asymptotic rate to the nonlinear diffusion wave $\chi(x, t)$ (see Corollary \ref{main.cor} below). Actually, an upper bound and a lower bound in the $L^{\infty}$-norm of the functions $Z(x, t)$ and $V(x, t)$ are obtained in \cite{F19-2, F20}. More precisely, there are strictly positive constants $\nu_{0}$ and $\nu_{1}$ such that 
\begin{align}
\begin{split}\label{Z-estimate}
\|Z(\cdot, t)\|_{L^{\infty}} 
&\begin{cases}
\le C\max\left\{|c_{\alpha}^{+}|, |c_{\alpha}^{-}|\right\}(1+t)^{-\frac{\alpha}{2}},  \\[.2em]
\ge \nu_{0}|\mu_{0}|(1+t)^{-\frac{\alpha}{2}}, &
\end{cases}
1<\alpha<2,
\end{split}\\
\begin{split}\label{Z+V-estimate}
\|Z(\cdot, t)+V(\cdot, t)\|_{L^{\infty}} 
&\begin{cases}
\le C(\max\left\{|c_{\alpha}^{+}|, |c_{\alpha}^{-}|\right\}+|\kappa d|\|V_{*}(\cdot)\|_{L^{\infty}})(1+t)^{-1}\log(1+t), & \\
\ge \nu_{1}|\mu_{1}|(1+t)^{-1}\log(1+t), &
\end{cases}
\alpha=2
\end{split}
\end{align}
hold for sufficiently large $t>0$, where $C$ is some positive constant and $\mu_0$ and $\mu_1$ are real constants given by
\begin{equation}
\mu_{0}:=\left(c_{\alpha}^{+}-c_{\alpha}^{-}\right)\Gamma\biggl(\frac{3-\alpha}{2}\biggl)+\frac{\left(c_{\alpha}^{+}+c_{\alpha}^{-}\right)\beta\chi_{*}(0)}{2-\alpha}\Gamma\biggl(2-\frac{\alpha}{2}\biggl), \quad
\mu_{1}:=\frac{c_{\alpha}^{+}+c_{\alpha}^{-}}{2}-\kappa d. 
\end{equation}
Here, $\Gamma$ means the Gamma function. On the other hand, from the definition of $V(x, t)$, i.e. \eqref{Second-AP-V}, we can easily have  
\begin{equation}\label{V-estimate}
\|V(\cdot, t)\|_{L^{\infty}}=|\kappa d|\|V_{*}(\cdot)\|_{L^{\infty}}(1+t)^{-1}\log(1+t).  
\end{equation}
Therefore, by using \eqref{main.thm-2nd.AP-1} through \eqref{main.thm-2nd.AP-3}, \eqref{Z-estimate}, \eqref{Z+V-estimate} and \eqref{V-estimate}, we have the following optimal estimates for $u-\chi$: 
\begin{cor}\label{main.cor}
Under the same assumptions in Theorem~\ref{main.thm-2nd.AP}, if $\mu_{0}\neq0$, $\mu_{1}\neq0$, $\kappa \neq0$ and $M\neq 0$, then we have
\begin{align}
\|u(\cdot, t)-\chi(\cdot, t)\|_{L^{\infty}} \sim \begin{cases}
(1+t)^{-\frac{\alpha}{2}}, &1<\alpha<2, \\
(1+t)^{-1}\log(1+t), &\alpha \ge2
\end{cases} 
\end{align}
for sufficiently large $t>0$.
\end{cor}

The rest of this paper is organized as follows. In Section~2, we prove the global existence and the $L^{p}$-decay estimates for the solutions to \eqref{BBMB}. In Section~3, first we prepare a couple of basic lemmas to prove our main theorems. After that, we prove Theorem~\ref{main.thm-1st.AP}, i.e. we show that the first asymptotic profile of the solutions to \eqref{BBMB} is given by the nonlinear diffusion wave $\chi(x, t)$. Finally, we give the proof of Theorem~\ref{main.thm-2nd.AP} in Section~4. This section is divided into three subsections. Subsection~4.1, which is for an auxiliary problem, is used to analyze the second asymptotic profiles. We prove Theorem~\ref{main.thm-2nd.AP} in Subsections~4.2 and 4.3. Subsection~4.2 is for the case of $1<\alpha <2$ and Subsection~4.3 is for the case of $\alpha=2$ and $\alpha>2$. Especially, the critical case $\alpha=2$ is the most difficult to treat, because it is necessary to consider effects both of the linear part and the Duhamel term with some higher-order derivatives.
\medskip

\par\noindent
\textbf{\bf{Notations.}} We define the Fourier transform of $f$ and the inverse Fourier transform of $g$ as follows:
\begin{align*}
\hat{f}(\xi):=\mathcal{F}[f](\xi)=\frac{1}{\sqrt{2\pi}}\int_{\R}e^{-ix\xi}f(x)dx, \quad 
\check{g}(x):=\mathcal{F}^{-1}[g](x)=\frac{1}{\sqrt{2\pi}}\int_{\R}e^{ix\xi}g(\xi)d\xi.
\end{align*}
For $1\le p \le \infty$, $L^{p}(\R)$ denotes the usual Lebesgue spaces. Also, we define a weighted Lebesgue space $L_{1}^{1}(\R)$ as follows 
\begin{equation*}
L^{1}_{1}(\R):=\left\{ f \in L^{1}(\R); \ \|f\|_{L^{1}_{1}}:=\int_{\R}|f(x)|(1+|x|)dx<\infty \right\}.
\end{equation*}
For a function $f\in L^{\infty}(\R)$ satisfying \eqref{data}, we define a norm $\|f\|_{L^{\infty}_{\alpha}}:=\text{ess.sup}_{x\in \R}|f(x)|(1+|x|)^{\alpha}$. 
\smallskip

For $1\le p\le \infty$ and an integer $m\ge 0$, we define the Sobolev spaces by 
\begin{equation*}
W^{m,p}(\R):=\left\{ f\in L^{p}(\R); \ \|f\|_{W^{m,p}}:=\left(\sum_{n=0}^{m}\left\| \p_{x}^{n}f\right\|_{L^{p}}^{p}\right)^{\frac{1}{p}}<\infty \right\}.
\end{equation*}
In a particular case $p=2$, we write $W^{m,2}(\R)=H^{m}(\R)$ with $\|f\|_{W^{m,p}}=\|f\|_{H^{m}}$. 
\smallskip

Let $I\subseteq [0, \infty)$ be an interval and $B$ be a Banach space. Then, for any non-negative integer $k\ge0$, $C^{k}(I; B)$ denotes the space of $B$-valued $k$-times continuously differentiable functions on $I$.
\smallskip

Throughout this paper, $C$ denotes various positive constants, which may vary from line to line during computations. Also, it may depend on the norm of the initial data. However, we note that it does not depend on the space variable $x$ and the time variable $t$. 

\smallskip
Finally, for positive functions $f(t)$ and $g(t)$ on time interval $I$, we write $f(t)\sim g(t)$ if there exist positive constants $c_{0}$ and $C_{0}$ independent of $t$ such that $c_{0}g(t)\le f(t)\le C_{0}g(t)$ holds. 


\section{Global Existence and Time Decay Estimates}

In this section, we shall prove the global existence and the time decay estimates for the solutions to the Cauchy problem~\eqref{BBMB} in Sobolev spaces. As was mentioned above, the global existence and appropriate time decay estimates in a lower order Sobolev space were already shown in \cite{ABS89, HKN07}. However, in order to analyze the second asymptotic profiles of the solution (Theorem \ref{main.thm-2nd.AP}), we need estimates in higher order Sobolev spaces. To obtain them, we modify the argument used in \cite{ABS89, HKN07} and generalize their results.

First, we consider the following linearized problem for \eqref{BBMB}: 
\begin{align*}
\tilde{u}_{t}-\tilde{u}_{xxt}-\tilde{u}_{xx}+\gamma \tilde{u}_{xxx}&=0, \ \ x\in \R, \ t>0,\\
\tilde{u}(x, 0)&=u_{0}(x) , \ \ x\in \R. 
\end{align*}
The solution to the above problem can be written by 
\[\tilde{u}(x, t)=(T(t)*u_{0})(x), \] 
where the integral kernel $T(x, t)$ is defined by 
\begin{equation}\label{linear-T}
T(x, t):=\mathcal{F}^{-1}\left[\frac{1}{\sqrt{2\pi}}\exp\left( \frac{-t\xi^{2}+i\gamma t\xi^{3} }{1+\xi^{2}}\right)\right](x). 
\end{equation}
For this solution to the linearized problem, we can show the following estimate: 
\begin{lem}\label{lem.linear-est}
Let $s$ be a non-negative integer. Suppose $f\in H^{s}(\R)\cap L^{1}(\R)$. Then, the estimate
\begin{equation}\label{linear-est}
\left\| \p^{l}_{x}T(t)*f\right\|_{L^{2}} \le C(1+t)^{-\frac{1}{4}-\frac{l}{2}}\|f\|_{L^{1}}+e^{-\frac{t}{2}}\left\| \p^{l}_{x}f\right\|_{L^{2}}, \ \ t\ge0
\end{equation}
holds for any integer $l$ satisfying $0\le l \le s$.
\end{lem}

\begin{rem}
{\rm 
A similar estimate is  given by Karch \cite{K99-1}. The different point is that our estimate improves his estimate near the origin $t=0$. 
}
\end{rem}

\begin{proof}
By using Plancherel's theorem, we have
\begin{align*}
\begin{split}
\left\|\p_{x}^{l}T(t)*f\right\|_{L^{2}}^{2}&=\left\|\exp\left( \frac{-t\xi^{2}+i\gamma t\xi^{3} }{1+\xi^{2}}\right)(i\xi)^{l}\hat{f}(\xi)\right\|_{L^{2}}^{2} \\
&=\left(\int_{|\xi|\le1}+\int_{|\xi|\ge1}\right)\exp\left( \frac{-2t\xi^{2}}{1+\xi^{2}}\right)\left|(i\xi)^{l}\hat{f}(\xi)\right|^{2}d\xi \\
&=:I_{1}+I_{2}.
\end{split}
\end{align*}
First, we evaluate $I_{1}$. Since $\exp\left( \frac{-2t\xi^{2}}{1+\xi^{2}}\right)\le e^{-t\xi^{2}}$ for all $|\xi|\le 1$, we have
\begin{align}
\label{linear-1}
\begin{split}
I_{1}&=\int_{|\xi|\le1}\exp\left( \frac{-2t\xi^{2}}{1+\xi^{2}}\right)\xi^{2l}|\hat{f}(\xi)|^{2}d\xi 
\le \int_{|\xi|\le 1}e^{-t\xi^{2}}\xi^{2l}|\hat{f}(\xi)|^{2}d\xi  \\
&\le 2\left(\sup_{|\xi| \le1}|\hat{f}(\xi)|\right)^{2}\int_{0}^{1}e^{-t\xi^{2}}\xi^{2l}d\xi  
=2\left(\sup_{|\xi| \le1}|\hat{f}(\xi)|\right)^{2}\int_{0}^{1}e^{\xi^{2}}e^{-(1+t)\xi^{2}}\xi^{2l}d\xi   \\
&\le 2e\left(\frac{\|f\|_{L^{1}}}{\sqrt{2\pi}}\right)^{2}\int_{0}^{1}e^{-(1+t)\xi^{2}}\xi^{2l}d\xi \le C(1+t)^{-\frac{1}{2}-l}\|f\|_{L^{1}}^{2}. 
\end{split}
\end{align}
Next, we evaluate $I_{2}$. By Plancherel's theorem, we have
\begin{align}
\label{linear-2}
\begin{split}
I_{2}\le \left(\sup_{|\xi|\ge1}\exp\left( \frac{-2t\xi^{2}}{1+\xi^{2}}\right)\right)\int_{|\xi|\ge1}\left|(i\xi)^{l}\hat{f}(\xi)\right|^{2}d\xi
\le e^{-t}\int_{\R}\left|\wh{\p_{x}^{l}f}(\xi)\right|^{2}d\xi 
=e^{-t}\left\|\p_{x}^{l}f\right\|_{L^{2}}^{2}.
\end{split}
\end{align}
From \eqref{linear-1} and \eqref{linear-2}, we obtain \eqref{linear-est}. This completes the proof. 
\end{proof}

Next, for the latter sake, we define the following nonlocal operator: 
\begin{equation}\label{nonlocal-op}
(1-\p_{x}^{2})^{-1}g(x):=\mathcal{F}^{-1}\left[\frac{1}{1+\xi^{2}}\hat{g}(\xi)\right](x)=\frac{1}{2}\int_{\R}e^{-|x-y|}g(y)dy=\frac{1}{2}(e^{-|\cdot|}*g)(x).
\end{equation}
From this definition, the following embedding theorem immediately follows.
\begin{lem}\label{lem.embedding}
Let $1\le p\le \infty$ and suppose $g\in L^{p}(\R)$. Then, we have 
\begin{equation}\label{est-embedding}
\left\| (1-\p_{x}^{2})^{-1}g\right\|_{L^{p}} \le C\|g\|_{L^{p}}. 
\end{equation}
\end{lem}

In addition to Lemma~\ref{lem.linear-est}, the following estimate is one of the key to prove the global existence and the time decay estimates for the solutions to \eqref{BBMB}.  
\begin{lem}\label{lem.Duhamel-est}
Let $s\ge1$ be an integer and suppose $g\in C^{0}((0, \infty); H^{s}(\R)) \cap C^{0}((0, \infty); W^{s, 1}(\R))$. Then, the estimate
\begin{align}\label{Duhamel-est}
\begin{split}
&\left\|\p_{x}^{l}\int_{0}^{t}(\p_{x}T(t-\tau))*\left((1-\p^{2}_{x})^{-1}g\right)(\tau)d\tau\right\|_{L^{2}} \\
&\le C\int_{0}^{t/2}(1+t-\tau)^{-\frac{3}{4}-\frac{l}{2}}\|g(\cdot, \tau)\|_{L^{1}}d\tau
+C\int_{t/2}^{t}(1+t-\tau)^{-\frac{3}{4}}\left\|\p_{x}^{l}g(\cdot, \tau)\right\|_{L^{1}}d\tau \\
&\ \ \ +C\left(\int_{0}^{t}e^{-\frac{t-\tau}{2}} \left\|\p_{x}^{l}g(\cdot, \tau)\right\|_{L^{2}}^{2}d\tau \right)^{\frac{1}{2}}, \ \ t>0
\end{split}
\end{align}
holds for any integer $l$ satisfying $0\le l \le s$, where $T(x, t)$ is defined by \eqref{linear-T}. 
\end{lem}
\begin{proof}
For simplicity, we set 
\begin{equation}\label{Duhamel-I}
I(x, t):=\int_{0}^{t}(\p_{x}T(t-\tau))*\left((1-\p_{x}^{2})^{-1}g\right)(\tau)d\tau. 
\end{equation}
By using Plancherel's theorem and splitting the integral, we have
\begin{align}
\label{I-split}
\begin{split}
\left\|\p_{x}^{l}I(\cdot, t)\right\|_{L^{2}} &\le \left\|(i\xi)^{l}\hat{I}(\xi, t)\right\|_{L^{2}(|\xi|\le1)}+\left\|(i\xi)^{l}\hat{I}(\xi, t)\right\|_{L^{2}(|\xi|\ge1)}=:I_{1}+I_{2}. 
\end{split}
\end{align}
Then, applying Lemma~\ref{lem.embedding}, similarly as \eqref{linear-1}, we obtain
\begin{align}
\label{Duhamel-I1-est}
\begin{split}
I_{1}&\le C\int_{0}^{t}\left\|(i\xi)^{l+1}\exp\left(\frac{-(t-\tau)\xi^{2}+i\gamma(t-\tau)\xi^{3}}{1+\xi^{2}}\right)\mathcal{F}\left[(1-\p_{x}^{2})^{-1}g\right](\xi, \tau)\right\|_{L^{2}(|\xi|\le1)}d\tau\\
&\le C\int_{0}^{t/2}\sup_{|\xi|\le1}\left|\mathcal{F}\left[(1-\p_{x}^{2})^{-1}g\right](\xi, \tau)\right|\left(\int_{|\xi|\le1}\xi^{2(l+1)}\exp\left(-\frac{2(t-\tau)\xi^{2}}{1+\xi^{2}}\right)d\xi \right)^{\frac{1}{2}}d\tau\\
&\ \ \ +C\int_{t/2}^{t}\sup_{|\xi|\le1}\left|(i\xi)^{l}\mathcal{F}\left[(1-\p_{x}^{2})^{-1}g\right](\xi, \tau)\right|\left(\int_{|\xi|\le1}\xi^{2}\exp\left(-\frac{2(t-\tau)\xi^{2}}{1+\xi^{2}}\right)d\xi \right)^{\frac{1}{2}}d\tau\\
&\le C\int_{0}^{t/2}\left\|(1-\p_{x}^{2})^{-1}g(\cdot, \tau)\right\|_{L^{1}}
\left(\int_{|\xi|\le1}\xi^{2(l+1)}e^{-(t-\tau)\xi^{2}}d\xi \right)^{\frac{1}{2}}d\tau \\
&\ \ \ +C\int_{t/2}^{t}\left\|(1-\p_{x}^{2})^{-1}\p_{x}^{l}g(\cdot, \tau)\right\|_{L^{1}}
\left(\int_{|\xi|\le1}\xi^{2}e^{-(t-\tau)\xi^{2}}d\xi \right)^{\frac{1}{2}}d\tau \\
&\le C\int_{0}^{t/2}(1+t-\tau)^{-\frac{3}{4}-\frac{l}{2}}\|g(\cdot, \tau)\|_{L^{1}}d\tau
+C\int_{t/2}^{t}(1+t-\tau)^{-\frac{3}{4}}\left\|\p_{x}^{l}g(\cdot, \tau)\right\|_{L^{1}}d\tau.
\end{split}
\end{align}
Here, we have used the following fact: 
\begin{equation*}
\int_{|\xi|\le1}|\xi|^{j}e^{-(t-\tau)|\xi|^{2}}d\xi \le C(1+t-\tau)^{-\frac{j}{2}-\frac{1}{2}}, \ \ j\ge0. 
\end{equation*}
Next, for $|\xi|\ge1$, by using Schwarz's inequality and \eqref{nonlocal-op}, we have 
\begin{align*}
\begin{split}
\left|(i\xi)^{l}\hat{I}(\xi, t)\right|
&=\left|(i\xi)^{l+1}\int_{0}^{t}\exp\left(\frac{-(t-\tau)\xi^{2}+i\gamma(t-\tau)\xi^{3}}{1+\xi^{2}}\right)\mathcal{F}\left[(1-\p_{x}^{2})^{-1}g\right](\xi, \tau)d\tau\right| \\
&\le \int_{0}^{t}\exp\left(-\frac{(t-\tau)\xi^{2}}{2(1+\xi^{2})}\right)\exp\left(-\frac{(t-\tau)\xi^{2}}{2(1+\xi^{2})}\right)\left|(i\xi)^{l+1}\mathcal{F}\left[(1-\p_{x}^{2})^{-1}g\right](\xi, \tau)\right|d\tau \\
&\le \left(\int_{0}^{t}\exp\left(-\frac{(t-\tau)\xi^{2}}{1+\xi^{2}}\right)d\tau \right)^{\frac{1}{2}}
\left(\int_{0}^{t}\exp\left(-\frac{(t-\tau)\xi^{2}}{1+\xi^{2}}\right)\left|\frac{(i\xi)^{l+1}}{1+\xi^{2}}\hat{g}(\xi, \tau)\right|^{2}d\tau \right)^{\frac{1}{2}} \\
&=\left\{\left(\frac{1+\xi^{2}}{\xi^{2}}\right)\left(1-\exp\left(\frac{-t\xi^{2}}{1+\xi^{2}}\right)\right)\right\}^{\frac{1}{2}} \\
&\ \ \ \times \left( \int_{0}^{t} \exp\left(-\frac{(t-\tau)\xi^{2}}{1+\xi^{2}}\right)  \frac{\xi^{2}}{1+\xi^{2}}\frac{1}{1+\xi^{2}} \left|(i\xi)^{l}\hat{g}(\xi, \tau)\right|^{2}d\tau \right)^{\frac{1}{2}}\\
&\le C\left(\int_{0}^{t}\exp\left(-\frac{(t-\tau)\xi^{2}}{1+\xi^{2}}\right) \left|(i\xi)^{l}\hat{g}(\xi, \tau)\right|^{2}d\tau \right)^{\frac{1}{2}}.
\end{split}
\end{align*}
Therefore, similarly as \eqref{linear-2}, we can see that 
\begin{align}
\label{Duhamel-I2-est}
\begin{split}
I_{2}&\le C\left(\int_{|\xi|\ge1}\int_{0}^{t}\exp\left(-\frac{(t-\tau)\xi^{2}}{1+\xi^{2}}\right) \left|(i\xi)^{l}\hat{g}(\xi, \tau)\right|^{2}d\tau d\xi\right)^{\frac{1}{2}} \\
&\le C\left(\int_{0}^{t}e^{-\frac{t-\tau}{2}}\int_{|\xi|\ge1}\left|(i\xi)^{l}\hat{g}(\xi, \tau)\right|^{2}d\xi d\tau \right)^{\frac{1}{2}} \\
&\le C\left(\int_{0}^{t}e^{-\frac{t-\tau}{2}}\left\|\p_{x}^{l}g(\cdot, \tau)\right\|_{L^{2}}^{2}d\tau \right)^{\frac{1}{2}}. 
\end{split}
\end{align}
Summarizing up \eqref{Duhamel-I} through \eqref{Duhamel-I2-est}, we obtain \eqref{Duhamel-est}. This completes the proof. 
\end{proof}

Now, let us prove the global existence and appropriate decay estimates of the solutions to \eqref{BBMB}. We give the proof of the following theorem by modifying the proof of Proposition~2.3 in \cite{F19-2}. 
\begin{thm}
\label{thm.SDGE}
Let $s\ge1$ be an integer. Assume that $u_{0}\in H^{s}(\R)\cap L^{1}(\R)$ and $\|u_{0}\|_{H^{s}}+\|u_{0}\|_{L^{1}}=:E_{s}$ is sufficiently small. Then, \eqref{BBMB} has a unique global mild solution $u\in C^{0}([0, \infty); H^{s}(\R))$. Moreover, the solution satisfies
\begin{equation}
\label{u-est-L2}
\left\| \p^{l}_{x}u(\cdot, t)\right\|_{L^{2}}\le CE_{s}(1+t)^{-\frac{1}{4}-\frac{l}{2}}, \ \ t\ge0 
\end{equation}
for any integer $l$ satisfying $0\le l \le s$. Furthermore, 
\begin{equation}
\label{u-est-Linfinity}
\left\| \p^{l}_{x}u(\cdot, t)\right\|_{L^{\infty}}\le CE_{s}(1+t)^{-\frac{1}{2}-\frac{l}{2}}, \ \ t\ge0
\end{equation}
holds for any integer $l$ satisfying $0\le l \le s-1$. 
\end{thm}
\begin{proof}
We consider the following integral equation associated with the Cauchy problem \eqref{BBMB}:
\begin{equation}
\label{integral-eq}
u(t)=T(t)*u_{0}-\frac{\beta}{2}\int_{0}^{t}(\p_{x}T(t-\tau))*\left((1-\p^{2}_{x})^{-1}(u^{2})\right)(\tau)d\tau. 
\end{equation}
We solve this integral equation by using the contraction mapping principle for the mapping 
\begin{equation}
\label{nonlinear-map}
N[u]:=T(t)*u_{0}-\frac{\beta}{2}\int_{0}^{t}(\p_{x}T(t-\tau))*\left((1-\p^{2}_{x})^{-1}(u^{2})\right)(\tau)d\tau. 
\end{equation}
Let us introduce the Banach space $X$ below 
\begin{equation}
\label{space-X}
X:=\left\{u\in C^{0}([0, \infty); H^{s}(\R)); \ \|u\|_{X}:=\sum_{l=0}^{s}\sup_{t\ge0}(1+t)^{\frac{1}{4}+\frac{l}{2}}\left\|\p_{x}^{l}u(\cdot, t)\right\|_{L^{2}}<\infty \right\}.
\end{equation}
We set $N_{0}:=T(t)*u_{0}$. Then, we have from Lemma \ref{lem.linear-est} that 
\begin{equation}
\label{est-N0}
\exists C_{0}>0 \ \ s.t. \ \ \|N_{0}\|_{X}\le C_{0}E_{s}.
\end{equation}
Now, we apply the contraction mapping principle to \eqref{nonlinear-map} on the closed subset $Y$ of $X$ below: 
\begin{equation*}
Y:=\{u\in X; \ \|u\|_{X}\le2C_{0}E_{s}\}.
\end{equation*}
Then, it is sufficient to show the following estimates: 
\begin{equation}
\label{contraction-1}
\|N[u]\|_{X}\le2C_{0}E_{s}, 
\end{equation}
\begin{equation}
\label{contraction-2}
\|N[u]-N[v]\|_{X}\le \frac{1}{2}\|u-v\|_{X} 
\end{equation}
for $u, v \in Y$. If we have shown \eqref{contraction-1} and \eqref{contraction-2}, by using the Banach fixed point theorem, we can prove that \eqref{integral-eq} has a unique global solution in $Y$, i.e. \eqref{BBMB} has a unique global mild solution. 

Here and later, $E_{s}$ is assumed to be sufficiently small. First, from Sobolev's inequality 
\begin{equation}\label{Sobolev-ineq}
\|f\|_{L^{\infty}}\le \sqrt{2}\|f\|_{L^{2}}^{\frac{1}{2}}\|f'\|_{L^{2}}^{\frac{1}{2}}, \ \ f\in H^{1}(\R), 
\end{equation}
we immediately have 
\begin{equation}
\label{u-est-X}
\left\|\p_{x}^{l}u(\cdot, t)\right\|_{L^{\infty}}\le \|u\|_{X}(1+t)^{-\frac{1}{2}-\frac{l}{2}}, \ \ 0\le l \le s-1.
\end{equation}
To prove \eqref{contraction-1} and \eqref{contraction-2}, we prepare the following estimates for $0\le l \le s$ and $u, v \in Y$:
\begin{align}
\label{prepare-est-1}
\left\|\p_{x}^{l}(u^{2}-v^{2})(\cdot, t)\right\|_{L^{1}}\le&C(\|u\|_{X}+\|v\|_{X})\|u-v\|_{X}(1+t)^{-\frac{1}{2}-\frac{l}{2}}, \\
\label{prepare-est-2}
\left\|\p_{x}^{l}(u^{2}-v^{2})(\cdot, t)\right\|_{L^{2}}\le&C(\|u\|_{X}+\|v\|_{X})\|u-v\|_{X}(1+t)^{-\frac{3}{4}-\frac{l}{2}}. 
\end{align}
Since we can prove \eqref{prepare-est-2} in the same way, we only prove \eqref{prepare-est-1}. We have from Schwarz's inequality and \eqref{space-X} that 
\begin{align*}
\begin{split}
\left\|\p_{x}^{l}(u^{2}-v^{2})(\cdot, t)\right\|_{L^{1}}&= \left\|\p_{x}^{l}((u+v)(u-v))(\cdot, t)\right\|_{L^{1}}\\
&\le C\sum_{m=0}^{l}\left(\left\|\p_{x}^{l-m}u(\cdot, t)\right\|_{L^{2}}+\left\|\p_{x}^{l-m}v(\cdot, t)\right\|_{L^{2}}\right)\left\|\p_{x}^{m}(u-v)(\cdot, t)\right\|_{L^{2}} \\
&\le C\sum_{m=0}^{l}(\|u\|_{X}+\|v\|_{X})(1+t)^{-\frac{1}{4}-\frac{l-m}{2}}\|u-v\|_{X}(1+t)^{-\frac{1}{4}-\frac{m}{2}}\\
&\le C(\|u\|_{X}+\|v\|_{X})\|u-v\|_{X}(1+t)^{-\frac{1}{2}-\frac{l}{2}}.
\end{split}
\end{align*}

Now, let us prove \eqref{contraction-1} and \eqref{contraction-2}. Recalling \eqref{nonlinear-map}, we get  
\begin{equation*}
(N[u]-N[v])(t)=-\frac{\beta}{2}\int_{0}^{t}(\p_{x}T(t-\tau))*\left((1-\p_{x}^{2})^{-1}(u^{2}-v^{2})\right)(\tau)d\tau. 
\end{equation*}
Therefore, applying Lemma~\ref{lem.Duhamel-est}, it follows from \eqref{prepare-est-1} and \eqref{prepare-est-2} that 
\begin{align*}
&\left\|\p_{x}^{l}(N[u]-N[v])(t)\right\|_{L^{2}} \\
&\le C\int_{0}^{t/2}(1+t-\tau)^{-\frac{3}{4}-\frac{l}{2}}\left\|(u^{2}-v^{2})(\cdot, \tau)\right\|_{L^{1}}d\tau
+C\int_{t/2}^{t}(1+t-\tau)^{-\frac{3}{4}}\left\|\p_{x}^{l}(u^{2}-v^{2})(\cdot, \tau)\right\|_{L^{1}}d\tau \\
&\ \ \ +C\left(\int_{0}^{t}e^{-\frac{t-\tau}{2}} \left\|\p_{x}^{l}(u^{2}-v^{2})(\cdot, \tau)\right\|_{L^{2}}^{2}d\tau \right)^{\frac{1}{2}} \\
&\le C(\|u\|_{X}+\|v\|_{X})\|u-v\|_{X}\biggl\{\int_{0}^{t/2}(1+t-\tau)^{-\frac{3}{4}-\frac{l}{2}}(1+\tau)^{-\frac{1}{2}}d\tau\\
&\ \ \ +\int_{t/2}^{t}(1+t-\tau)^{-\frac{3}{4}}(1+\tau)^{-\frac{1}{2}-\frac{l}{2}}d\tau+\left(\int_{0}^{t}e^{-\frac{t-\tau}{2}}(1+\tau)^{-\frac{3}{2}-l}d\tau \right)^{\frac{1}{2}}\biggl\}\\
&\le C(\|u\|_{X}+\|v\|_{X})\|u-v\|_{X}\left\{(1+t)^{-\frac{1}{4}-\frac{l}{2}}+(1+t)^{-\frac{3}{4}-\frac{l}{2}}\right\} \\
&\le C(\|u\|_{X}+\|v\|_{X})\|u-v\|_{X}(1+t)^{-\frac{1}{4}-\frac{l}{2}}, \ \ 0\le l\le s. 
\end{align*}
Thus, there exists a positive constant $C_{1}>0$ such that 
\begin{equation*}
\|N[u]-N[v]\|_{X}\le C_{1}(\|u\|_{X}+\|v\|_{X})\|u-v\|_{X}\le 4C_{0}C_{1}E_{s}\|u-v\|_{X}, \ \ u, v \in Y.
\end{equation*}
Choosing $E_{s}$ which satisfies $4C_{0}C_{1}E_{s} \le1/2$, then we have \eqref{contraction-2}. 
Moreover, taking $v=0$ in \eqref{contraction-2}, we can immediately see that 
\begin{equation*}
\|N[u]-N[0]\|_{X}\le C_{0}E_{s}.
\end{equation*}
Since $N[0]=N_{0}$, we obtain from \eqref{est-N0} that 
\begin{equation*}
\|N[u]\|_{X}\le \|N_{0}\|_{X}+\|N[u]-N[0]\|_{X} \le 2C_{0}E_{s}.
\end{equation*}
Therefore, we get \eqref{contraction-1}. This completes the proof of the global existence and of the $L^{2}$-decay estimate \eqref{u-est-L2} for the solutions to \eqref{BBMB}. 
The $L^{\infty}$-decay estimate \eqref{u-est-Linfinity} directly follows from \eqref{u-est-L2} through Sobolev's inequality \eqref{Sobolev-ineq}.  
\end{proof}

\section{First Asymptotic Profile}

The purpose of this section is to prove Theorem~\ref{main.thm-1st.AP}, namely, we shall show that the first asymptotic profile of the solutions to~\eqref{BBMB} is given by $\chi(x, t)$ defined by~\eqref{diffusion-wave}. First of all, to discuss the asymptotic behavior, we introduce the following two basic lemmas. The first one is the $L^{p}$-decay estimate for $\chi(x, t)$ (for the proof, see Lemma~4.3 in \cite{KU17}). 
\begin{lem}\label{lem.diffusion-decay}
Let $k$ and $l$ be non-negative integers. Then, for $|M|\le1$ and $1\le p\le \infty$, we have
\begin{equation}\label{diffusion-decay}
\left\| \p_{t}^{k}\p_{x}^{l}\chi(\cdot, t)\right\|_{L^{p}}\le C|M|(1+t)^{-\frac{1}{2}(1-\frac{1}{p})-\frac{l}{2}-k}, \ \ t\ge0.
\end{equation}
\end{lem}
\noindent 
Next, we introduce the $L^{p}$-decay estimate for the heat kernel $G(x, t)$ defined by \eqref{heat-eta}, and the estimate for the convolution $G(t)*\phi$ (for the proof, see Lemma~7.1 in~\cite{UK07} and Lemma~2.4 in \cite{F20}).
\begin{lem}\label{lem.heat-decay}
Let $k$ and $l$ be non-negative integers. Then, for $1\le p\le \infty$, we have
\begin{equation}\label{heat-decay}
\left\| \p_{t}^{k}\p_{x}^{l}G(\cdot, t)\right\|_{L^{p}}\le Ct^{-\frac{1}{2}(1-\frac{1}{p})-\frac{l}{2}-k}, \ \ t>0.
\end{equation}
Moreover, if $\phi \in L^{\infty}(\R)$ satisfies \eqref{data} and $\int_{\R}\phi(x)dx=0$, then the following estimate holds: 
\begin{align}\label{heat-decay-slow}
\left\|\p_{t}^{k}\p_{x}^{l}G(t)*\phi \right\|_{L^{p}}\le C\begin{cases}
t^{-\frac{\alpha}{2}+\frac{1}{2p}-\frac{l}{2}-k}\|\phi\|_{L^{\infty}_{\alpha}}, &t>0, \ 1<\alpha<2,\\
t^{-1+\frac{1}{2p}-\frac{l}{2}-k}\log(2+t)\|\phi\|_{L^{\infty}_{\alpha}}, &t>0, \ \alpha=2, \\
t^{-1+\frac{1}{2p}-\frac{l}{2}-k}\|\phi\|_{L^{1}_{1}}, &t>0, \ \alpha>2.
\end{cases}
\end{align}
\end{lem}

Now, let us discuss the asymptotic behavior. First, recalling the integral equation:
\begin{equation}
\tag{\ref{integral-eq}}
u(t)=T(t)*u_{0}-\frac{\beta}{2}\int_{0}^{t}(\p_{x}T(t-\tau))*\left((1-\p^{2}_{x})^{-1}(u^{2})\right)(\tau)d\tau. 
\end{equation}
In addition, applying the Duhamel principle to \eqref{Burgers}, we obtain 
\begin{equation}
\label{integral-eq-B}
\chi(t)=G(t)*\chi_{0}-\frac{\beta}{2}\int_{0}^{t}(\p_{x}G(t-\tau))*(\chi^{2})(\tau)d\tau, 
\end{equation}
where $\chi_{0}(x):=\chi(x, 0)=\chi_{*}(x)$. Here $\chi_*(x)$ is defined by \eqref{diffusion-wave*}. Now, we set 
\begin{equation}\label{difference}
\psi(x, t):=u(x, t)-\chi(x, t), \quad \psi_{0}(x):=u_{0}(x)-\chi_{*}(x). 
\end{equation}
Then, from \eqref{integral-eq}, \eqref{integral-eq-B} and \eqref{difference}  and by using $(1-\p_{x}^{2})^{-1}(\chi^{2})=\chi^{2}+(1-\p_{x}^{2})^{-1}\p_{x}^{2}(\chi^{2})$, we have
\begin{align}
\label{integral-eq-psi}
\begin{split}
\psi(t)&=(T-G)(t)*u_{0}+G(t)*\psi_{0} \\
&\ \ \ \ -\frac{\beta}{2}\int_{0}^{t}(\p_{x}T(t-\tau))*\left((1-\p^{2}_{x})^{-1}(u^{2}-\chi^{2})\right)(\tau)d\tau  \\
&\ \ \ \ -\frac{\beta}{2}\int_{0}^{t}(\p_{x}(T-G)(t-\tau))*(\chi^{2})(\tau)d\tau \\
&\ \ \ \ -\frac{\beta}{2}\int_{0}^{t}(\p_{x}T(t-\tau))*\left((1-\p^{2}_{x})^{-1}\p_{x}^{2}(\chi^{2})\right)(\tau)d\tau  \\
&=:I_{1}+I_{2}+I_{3}+I_{4}+I_{5}. 
\end{split}
\end{align}

Our first step to prove Theorem~\ref{main.thm-1st.AP} is to derive the following proposition: 
\begin{prop}\label{prop.linear-asymptotic}
Let $s$ be a non-negative integer. Suppose $f\in H^{s}(\R) \cap L^{1}(\R)$. Then, the estimate
\begin{equation}\label{linear-est-asymptotic}
\left\| \p^{l}_{x}(T-G)(t)*f\right\|_{L^{2}} \le C(1+t)^{-\frac{3}{4}-\frac{l}{2}}\|f\|_{L^{1}}+Ce^{-\frac{t}{2}}\left\| \p^{l}_{x}f\right\|_{L^{2}}, \ \ t\ge0
\end{equation}
holds for any integer $l$ satisfying $0\le l \le s$, where $T(x, t)$ and $G(x, t)$ are defined by \eqref{linear-T} and \eqref{heat-eta}, respectively.
\end{prop}
\begin{proof}
By using Plancherel's theorem, we obtain
\begin{align}\label{linear-ap-T-G}
\begin{split}
&\left\|\p_{x}^{l}(T-G)(t)*f\right\|_{L^{2}}^{2} \\
&=\left\|\left\{\exp\left( \frac{-t\xi^{2}+i\gamma t\xi^{3} }{1+\xi^{2}}\right)-e^{-t\xi^{2}}\right\}(i\xi)^{l}\hat{f}(\xi)\right\|_{L^{2}}^{2} \\
&=\left\|\exp\left( \frac{-t\xi^{2}}{1+\xi^{2}}\right)\left\{ \exp\left( \frac{i\gamma t\xi^{3} }{1+\xi^{2}}\right)-\exp\left( \frac{-t\xi^{4}}{1+\xi^{2}}\right)\right\}(i\xi)^{l}\hat{f}(\xi)\right\|_{L^{2}}^{2} \\
&=\left(\int_{|\xi|\le1}+\int_{|\xi|\ge1}\right)\exp\left( \frac{-2t\xi^{2}}{1+\xi^{2}}\right)J(\xi, t)\left|(i\xi)^{l}\hat{f}(\xi)\right|^{2}d\xi \\
&=:K_{1}+K_{2}, 
\end{split}
\end{align}
where we have defined $J(\xi, t)$ as follows 
\begin{equation}\label{linear-ap-J}
J(\xi, t):=\left| \exp\left( \frac{i\gamma t\xi^{3} }{1+\xi^{2}}\right)-\exp\left( \frac{-t\xi^{4}}{1+\xi^{2}}\right)\right|^{2}. 
\end{equation} 
From the mean value theorem, there exist $\theta_{1}=\theta_{1}(\xi, t, \gamma)$ and $\theta_{2}=\theta_{2}(\xi, t)$ such that 
\begin{align}
\exp\left( \frac{i\gamma t\xi^{3} }{1+\xi^{2}}\right)-1
&=\frac{i\gamma t\xi^{3}}{1+\xi^{2}}\exp\left( \frac{i\gamma \theta_{1}t\xi^{3} }{1+\xi^{2}}\right), \label{linear-ap-mean1}\\
1-\exp\left( \frac{-t\xi^{4}}{1+\xi^{2}}\right)
&=\frac{t\xi^{4}}{1+\xi^{2}}\exp\left( \frac{-\theta_{2}t\xi^{4}}{1+\xi^{2}}\right). \label{linear-ap-mean2}
\end{align}
Here, we note that $0<\theta_{1}, \theta_{2}<1$. Therefore, it follows from \eqref{linear-ap-J}, \eqref{linear-ap-mean1} and \eqref{linear-ap-mean2} that 
\begin{align}\label{linear-ap-J-est}
\begin{split}
J(\xi, t)&=\left| \exp\left( \frac{i\gamma t\xi^{3} }{1+\xi^{2}}\right)-1+1-\exp\left( \frac{-t\xi^{4}}{1+\xi^{2}}\right)\right|^{2} \\
&\le 2\left| \exp\left( \frac{i\gamma t\xi^{3} }{1+\xi^{2}}\right)-1\right|^{2}+2\left|1-\exp\left( \frac{-t\xi^{4}}{1+\xi^{2}}\right)\right|^{2} \\
&\le 2t^{2}|\gamma|^{2}\left|\frac{\xi^{3}}{1+\xi^{2}}\right|^{2}+2t^{2}\left|\frac{\xi^{4}}{1+\xi^{2}}\right|^{2}\exp\left( \frac{-2\theta_{2}t\xi^{4}}{1+\xi^{2}}\right) \\
&\le 2t^{2}\left(|\gamma|^{2}\xi^{6}+\xi^{8}\right). 
\end{split}
\end{align}

Now, let us evaluate $K_{1}$ and $K_{2}$. First, we shall treat $K_{1}$. By using \eqref{linear-ap-J-est}, in the same way to get \eqref{linear-1}, we can see that 
\begin{align}
\label{linear-ap-K1}
\begin{split}
K_{1}&=2t^{2} \int_{|\xi|\le1}\exp\left( \frac{-2t\xi^{2}}{1+\xi^{2}}\right)\left(|\gamma|^{2}\xi^{6}+\xi^{8}\right)\xi^{2l}|\hat{f}(\xi)|^{2}d\xi  \\
&\le 4\max\{|\gamma|^{2}, 1\}\,t^{2} \int_{|\xi|\le 1}e^{-t\xi^{2}}\xi^{2(l+3)}|\hat{f}(\xi)|^{2}d\xi  \\
&\le 8\max\{|\gamma|^{2}, 1\}\,t^{2} \left(\sup_{|\xi| \le1}|\hat{f}(\xi)|\right)^{2}\int_{0}^{1}e^{-t\xi^{2}}\xi^{2(l+3)}d\xi   \\
&=8\max\{|\gamma|^{2}, 1\}\,t^{2} \left(\sup_{|\xi| \le1}|\hat{f}(\xi)|\right)^{2}\int_{0}^{1}e^{\xi^{2}}e^{-(1+t)\xi^{2}}\xi^{2(l+3)}d\xi   \\
&\le 8e \max\{|\gamma|^{2}, 1\}\,t^{2} \left(\frac{\|f\|_{L^{1}}}{\sqrt{2\pi}}\right)^{2}\int_{0}^{1}e^{-(1+t)\xi^{2}}\xi^{2(l+3)}d\xi \\
&\le C(1+t)^{-\frac{1}{2}-(l+3)+2}\|f\|_{L^{1}}^{2}
=C(1+t)^{-\frac{3}{2}-l}\|f\|_{L^{1}}^{2}. 
\end{split}
\end{align}
Finally, we evaluate $K_{2}$. To do that, instead of \eqref{linear-ap-J-est}, we shall use the boundedness of $J(\xi, t)$: 
\[J(\xi, t)\le \left( \left|\exp\left( \frac{i\gamma t\xi^{3} }{1+\xi^{2}}\right)\right| + \left|\exp\left( \frac{-t\xi^{4}}{1+\xi^{2}}\right)\right|   \right)^{2}
\le 4, \quad \xi \in \R, \ t\ge0. \]
Therefore, similarly as \eqref{linear-2}, we have
\begin{equation}\label{linear-ap-K2}
K_{2}\le 4\left(\sup_{|\xi|\ge1}\exp\left( \frac{-2t\xi^{2}}{1+\xi^{2}}\right)\right)\int_{|\xi|\ge1}\left|(i\xi)^{l}\hat{f}(\xi)\right|^{2}d\xi
\le 4e^{-t}\left\|\p_{x}^{l}f\right\|_{L^{2}}^{2}.
\end{equation}
From \eqref{linear-ap-K1} and \eqref{linear-ap-K2}, we obtain \eqref{linear-est-asymptotic}. This completes the proof. 
\end{proof}

In order to prove Theorem~\ref{main.thm-1st.AP}, it suffices to show the following $L^{2}$-decay estimate of $\psi(x, t)$: 
\begin{prop}\label{prop.psi-est-L2}
Let $s\ge1$ be an integer. Assume that the initial data $u_{0}(x)$ satisfies the condition \eqref{data}, $u_{0}\in H^{s}(\R)$ and $\|u_{0}\|_{H^{s}}+\|u_{0}\|_{L^{1}}$ is sufficiently small. Then, for any $\e>0$, the estimate 
\begin{align}\label{psi-est-L2}
\left\|\p_{x}^{l}\psi(\cdot, t)\right\|_{L^{2}}\le C\begin{cases}
(1+t)^{-\frac{\alpha}{2}+\frac{1}{4}-\frac{l}{2}}, &t\ge0, \ 1<\alpha<2,\\
(1+t)^{-\frac{3}{4}-\frac{l}{2}+\e}, &t\ge0, \ \alpha\ge2
\end{cases}
\end{align}
holds for any integer $l$ satisfying $0\le l\le s$, where $\psi(x, t)$ is defined by \eqref{difference}.
\end{prop} 
\begin{proof}
We evaluate $\psi(x, t)$ by introducing the following quantity:  
\begin{align}
\label{M(T)}
M_{\alpha}(T):=\begin{cases}\d
\sup_{0\le t\le T}\sum_{l=0}^{s}(1+t)^{\frac{\alpha}{2}-\frac{1}{4}+\frac{l}{2}}\left\|\p_{x}^{l}\psi(\cdot, t)\right\|_{L^{2}}, &1<\alpha<2,\\
\d \sup_{0\le t\le T}\sum_{l=0}^{s}(1+t)^{\frac{3}{4}+\frac{l}{2}-\e}\left\|\p_{x}^{l}\psi(\cdot, t)\right\|_{L^{2}}, &\alpha\ge2, 
\end{cases}
\end{align}
where $\e$ is any fixed constant such that $0<\e<3/4$. To prove \eqref{psi-est-L2}, it suffices to estimate the each term of the right hand side of \eqref{integral-eq-psi}. For the first term, by using Proposition~\ref{prop.linear-asymptotic}, we obtain 
\begin{equation}\label{1st.AP-I1-est}
\left\|\p_{x}^{l}I_{1}(\cdot, t)\right\|_{L^{2}}\le C(\|u_{0}\|_{H^{s}}+\|u_{0}\|_{L^{1}})(1+t)^{-\frac{3}{4}-\frac{l}{2}}, \ \ t\ge0.
\end{equation}
Also, since $\int_{\R}\psi_{0}(x)dx=\int_{\R}(u_{0}(x)-\chi_{*}(x))dx=0$ from \eqref{diffusion-wave*}, it follows from \eqref{data} and \eqref{heat-decay-slow} that 
\begin{align}
\label{1st.AP-I2-est}
\left\|\p_{x}^{l}I_{2}(\cdot, t)\right\|_{L^{2}}\le C\begin{cases}
(1+t)^{-\frac{\alpha}{2}+\frac{1}{4}-\frac{l}{2}}\|\psi_{0}\|_{L^{\infty}_{\alpha}}, &t\ge1, \ 1<\alpha<2,\\
(1+t)^{-\frac{3}{4}-\frac{l}{2}}\log(2+t)\|\psi_{0}\|_{L^{\infty}_{\alpha}}, &t\ge1, \ \alpha=2, \\
(1+t)^{-\frac{3}{4}-\frac{l}{2}}\|\psi_{0}\|_{L^{1}_{1}}, &t\ge1, \ \alpha>2. 
\end{cases}
\end{align}
Here we have used the fact that for $\chi_{*}(x)$, $|\chi_{*}(x)|\le C|M|e^{-\frac{x^{2}}{4}}\le C|M|(1+|x|)^{-N}$ for all $N\in \mathbb{N}$.

Next, we shall evaluate the Duhamel terms $I_{3}$, $I_{4}$ and $I_{5}$ in \eqref{integral-eq-psi}. Before doing that, we prepare the following two estimates for $0\le l\le s$: 
\begin{align}
\left\|\p_{x}^{l}(u^{2}-\chi^{2})(\cdot, t)\right\|_{L^{1}}&
\le CE_{s}M_{\alpha}(T)\begin{cases}
(1+t)^{-\frac{\alpha}{2}-\frac{l}{2}}, &1<\alpha<2,\\
(1+t)^{-1-\frac{l}{2}+\e}, &\alpha\ge2, 
\end{cases} \label{u^2-chi^2-L1}\\
\left\|\p_{x}^{l}(u^{2}-\chi^{2})(\cdot, t)\right\|_{L^{2}}&
\le CE_{s}M_{\alpha}(T)\begin{cases}
(1+t)^{-\frac{\alpha}{2}-\frac{1}{4}-\frac{l}{2}}, &1<\alpha<2,\\
(1+t)^{-\frac{5}{4}-\frac{l}{2}+\e}, &\alpha\ge2. 
\end{cases}\label{u^2-chi^2-L2}
\end{align}
Let $0\le l\le s$ and $0\le t\le T$. From Theorem~\ref{thm.SDGE}, Lemma~\ref{lem.diffusion-decay} and \eqref{M(T)}, we have 
\begin{align*}
\left\|\p_{x}^{l}(u^{2}-\chi^{2})(\cdot, t)\right\|_{L^{1}}
&=\left\|\p_{x}^{l}((u+\chi)\psi)(\cdot, t)\right\|_{L^{1}} \\
&\le C\sum_{m=0}^{l}\left\|\p_{x}^{l-m}(u+\chi)(\cdot, t)\right\|_{L^{2}}\left\|\p_{x}^{m}\psi(\cdot, t)\right\|_{L^{2}} \\
&\le C\sum_{m=0}^{l}\left(\left\|\p_{x}^{l-m}u(\cdot, t)\right\|_{L^{2}}+\left\|\p_{x}^{l-m}\chi(\cdot, t)\right\|_{L^{2}}\right)\left\|\p_{x}^{m}\psi(\cdot, t)\right\|_{L^{2}} \\
&\le CE_{s}M_{\alpha}(T)\sum_{m=0}^{l}(1+t)^{-\frac{1}{4}-\frac{l-m}{2}}\begin{cases}
(1+t)^{-\frac{\alpha}{2}+\frac{1}{4}-\frac{m}{2}}, &1<\alpha<2,\\
(1+t)^{-\frac{3}{4}-\frac{m}{2}+\e}, &\alpha\ge2
\end{cases}\\
&\le CE_{s}M_{\alpha}(T)\begin{cases}
(1+t)^{-\frac{\alpha}{2}-\frac{l}{2}}, &1<\alpha<2,\\
(1+t)^{-1-\frac{l}{2}+\e}, &\alpha\ge2
\end{cases}
\end{align*}
and
\begin{align*}
\left\|\p_{x}^{l}(u^{2}-\chi^{2})(\cdot, t)\right\|_{L^{2}}
&=\left\|\p_{x}^{l}((u+\chi)\psi)(\cdot, t)\right\|_{L^{2}} \\
&\le C\left(\left\|\p_{x}^{l}u(\cdot, t)\right\|_{L^{2}}+\left\|\p_{x}^{l}\chi(\cdot, t)\right\|_{L^{2}}\right)\|\psi(\cdot, t)\|_{L^{\infty}} \\
&\ \ \ +C\sum_{m=0}^{l-1}\left(\left\|\p_{x}^{m}u(\cdot, t)\right\|_{L^{\infty}}+\left\|\p_{x}^{m}\chi(\cdot, t)\right\|_{L^{\infty}}\right)\left\|\p_{x}^{l-m}\psi(\cdot, t)\right\|_{L^{2}} \\
&\le CE_{s}M_{\alpha}(T)(1+t)^{-\frac{1}{4}-\frac{l}{2}}\begin{cases}
(1+t)^{-\frac{\alpha}{2}}, &1<\alpha<2,\\
(1+t)^{-1+\e}, &\alpha\ge2
\end{cases}\\
&\ \ \ +CE_{s}M_{\alpha}(T)\sum_{m=0}^{l-1}(1+t)^{-\frac{1}{2}-\frac{m}{2}}\begin{cases}
(1+t)^{-\frac{\alpha}{2}+\frac{1}{4}-\frac{l-m}{2}}, &1<\alpha<2,\\
(1+t)^{-\frac{3}{4}-\frac{l-m}{2}+\e}, &\alpha\ge2
\end{cases}\\
&\le CE_{s}M_{\alpha}(T)\begin{cases}
(1+t)^{-\frac{\alpha}{2}-\frac{1}{4}-\frac{l}{2}}, &1<\alpha<2,\\
(1+t)^{-\frac{5}{4}-\frac{l}{2}+\e}, &\alpha\ge2.
\end{cases}
\end{align*}
Now, applying Lemma~\ref{lem.Duhamel-est} to $I_{3}$, we obtain 
\begin{align}
&\left\|\p_{x}^{l}I_{3}(\cdot, t)\right\|_{L^{2}}
\le C\int_{0}^{t/2}(1+t-\tau)^{-\frac{3}{4}-\frac{l}{2}}\left\|(u^{2}-\chi^{2})(\cdot, \tau)\right\|_{L^{1}}d\tau \nonumber \\
&+C\int_{t/2}^{t}(1+t-\tau)^{-\frac{3}{4}}\left\|\p_{x}^{l}(u^{2}-\chi^{2})(\cdot, \tau)\right\|_{L^{1}}d\tau
+C\left(\int_{0}^{t}e^{-\frac{t-\tau}{2}} \left\|\p_{x}^{l}(u^{2}-\chi^{2})(\cdot, \tau)\right\|_{L^{2}}^{2}d\tau \right)^{\frac{1}{2}} \nonumber \\
&=:I_{3.1}+I_{3.2}+I_{3.3}. \label{1st.AP-I3-est*}
\end{align}
By using \eqref{u^2-chi^2-L1} and \eqref{u^2-chi^2-L2}, we can see that the following estimates holds: 
\begin{align}\label{1st.AP-I3.1-est}
\begin{split}
I_{3.1}&\le CE_{s}M_{\alpha}(T)(1+t)^{-\frac{3}{4}-\frac{l}{2}}
\begin{cases}
\d \int_{0}^{t/2}(1+\tau)^{-\frac{\alpha}{2}}d\tau, &1<\alpha<2,\\[3mm]
\d \int_{0}^{t/2}(1+\tau)^{-1+\e}d\tau, &\alpha\ge2 
\end{cases}\\
&\le CE_{s}M_{\alpha}(T)\begin{cases}
(1+t)^{-\frac{\alpha}{2}+\frac{1}{4}-\frac{l}{2}}, &1<\alpha<2,\\
(1+t)^{-\frac{3}{4}-\frac{l}{2}+\e}, &\alpha\ge2, 
\end{cases}
\end{split}
\end{align}
\begin{align}\label{1st.AP-I3.2-est}
\begin{split}
I_{3.2}&\le CE_{s}M_{\alpha}(T)
\begin{cases}
\d \int_{t/2}^{t}(1+t-\tau)^{-\frac{3}{4}}(1+\tau)^{-\frac{\alpha}{2}-\frac{l}{2}}d\tau, &1<\alpha<2,\\[3mm]
\d \int_{t/2}^{t}(1+t-\tau)^{-\frac{3}{4}}(1+\tau)^{-1-\frac{l}{2}+\e}d\tau, &\alpha\ge2 
\end{cases}\\
&\le CE_{s}M_{\alpha}(T)\begin{cases}
(1+t)^{-\frac{\alpha}{2}+\frac{1}{4}-\frac{l}{2}}, &1<\alpha<2,\\
(1+t)^{-\frac{3}{4}-\frac{l}{2}+\e}, &\alpha\ge2
\end{cases}
\end{split}
\end{align}
and 
\begin{align}\label{1st.AP-I3.3-est}
\begin{split}
I_{3.3}&\le CE_{s}M_{\alpha}(T)
\begin{cases}
\d \left(\int_{0}^{t}e^{-\frac{t-\tau}{2}}(1+\tau)^{-\alpha-\frac{1}{2}-l}d\tau\right)^{\frac{1}{2}}, &1<\alpha<2,\\
\d \left(\int_{0}^{t}e^{-\frac{t-\tau}{2}}(1+\tau)^{-\frac{5}{2}-l+2\e}d\tau\right)^{\frac{1}{2}}, &\alpha\ge2 
\end{cases}\\
&\le CE_{s}M_{\alpha}(T)\begin{cases}
(1+t)^{-\frac{\alpha}{2}-\frac{1}{4}-\frac{l}{2}}, &1<\alpha<2,\\
(1+t)^{-\frac{5}{4}-\frac{l}{2}+\e}, &\alpha\ge2. 
\end{cases}
\end{split}
\end{align}
Therefore, summarizing up \eqref{1st.AP-I3-est*} through \eqref{1st.AP-I3.3-est}, we obtain 
\begin{equation}\label{1st.AP-I3-est}
\left\|\p_{x}^{l}I_{3}(\cdot, t)\right\|_{L^{2}}\le CE_{s}M_{\alpha}(T)\begin{cases}
(1+t)^{-\frac{\alpha}{2}+\frac{1}{4}-\frac{l}{2}}, &t\ge0, \ 1<\alpha<2,\\
(1+t)^{-\frac{3}{4}-\frac{l}{2}+\e}, &t\ge0, \ \alpha\ge2. 
\end{cases}
\end{equation}

Next, we deal with the estimate for $I_{4}$. It follows from Proposition~\ref{prop.linear-asymptotic} and Lemma~\ref{lem.diffusion-decay} that 
\begin{align}
\label{1st.AP-I4-est}
\begin{split}
&\left\|\p_{x}^{l}I_{4}(\cdot, t)\right\|_{L^{2}}
\le C\int_{0}^{t/2}\left\|(\p_{x}^{l+1}(T-G)(t-\tau))*(\chi^{2})(\tau)\right\|_{L^{2}}d\tau\\
&\ \ \ +C\int_{t/2}^{t}\left\|(T-G)(t-\tau)*(\p_{x}^{l+1}(\chi^{2})(\tau)) \right\|_{L^{2}}d\tau\\
&\le C\int_{0}^{t/2}(1+t-\tau)^{-\frac{3}{4}-\frac{l+1}{2}}\left\|\chi^{2}(\cdot, \tau)\right\|_{L^{1}}d\tau
+C\int_{t/2}^{t}(1+t-\tau)^{-\frac{3}{4}}\left\|\p_{x}^{l+1}(\chi^{2}(\cdot, \tau))\right\|_{L^{1}}d\tau \\
&\ \ \ +C\int_{0}^{t}e^{-\frac{t-\tau}{2}}\left\|\p_{x}^{l+1}(\chi^{2}(\cdot, \tau))\right\|_{L^{2}}d\tau  \\
&\le CE_{s}\int_{0}^{t/2}(1+t-\tau)^{-\frac{5}{4}-\frac{l}{2}}(1+\tau)^{-\frac{1}{2}}d\tau
+CE_{s}\int_{t/2}^{t}(1+t-\tau)^{-\frac{3}{4}}(1+\tau)^{-1-\frac{l}{2}}d\tau \\
&\ \ \ +CE_{s}\int_{0}^{t}e^{-\frac{t-\tau}{2}}(1+\tau)^{-\frac{3}{4}-\frac{l}{2}}d\tau \\
&\le CE_{s}(1+t)^{-\frac{3}{4}-\frac{l}{2}}, \ \ t\ge0, 
\end{split}
\end{align}
where we have used the estimate $\left\|\p_{x}^{l}(\chi^{2}(\cdot, t))\right\|_{L^{p}}\le CE_{s}(1+t)^{-1+\frac{1}{2p}-\frac{l}{2}}$ for $1\le p\le \infty$. 

Finally, we shall evaluate $I_{5}$. Applying Lemma~\ref{lem.Duhamel-est} to $I_{5}$ again, similarly as before, we obtain 
\begin{align}\label{1st.AP-I5-est}
\begin{split}
&\left\|\p_{x}^{l}I_{5}(\cdot, t)\right\|_{L^{2}}
\le C\int_{0}^{t/2}(1+t-\tau)^{-\frac{3}{4}-\frac{l}{2}}\left\|\p_{x}^{2}(\chi^{2}(\cdot, \tau))\right\|_{L^{1}}d\tau \\
&\ \ \ +C\int_{t/2}^{t}(1+t-\tau)^{-\frac{3}{4}}\left\|\p_{x}^{l+2}(\chi^{2}(\cdot, \tau))\right\|_{L^{1}}d\tau
+C\left(\int_{0}^{t}e^{-\frac{t-\tau}{2}} \|\p_{x}^{l+2}(\chi^{2}(\cdot, \tau))\|_{L^{2}}^{2}d\tau \right)^{\frac{1}{2}}\\
&\le C\int_{0}^{t/2}(1+t-\tau)^{-\frac{3}{4}-\frac{l}{2}}(1+\tau)^{-\frac{3}{2}}d\tau
+C\int_{t/2}^{t}(1+t-\tau)^{-\frac{3}{4}}(1+\tau)^{-\frac{3}{2}-\frac{l}{2}}d\tau  \\
&\ \ \ +C\left(\int_{0}^{t}e^{-\frac{t-\tau}{2}}(1+\tau)^{-\frac{7}{2}-l}d\tau \right)^{\frac{1}{2}}\\
&\le CE_{s}(1+t)^{-\frac{3}{4}-\frac{l}{2}}, \ \ t\ge0. 
\end{split}
\end{align} 

Summarizing up \eqref{integral-eq-psi}, \eqref{1st.AP-I1-est}, \eqref{1st.AP-I2-est}, \eqref{1st.AP-I3-est}, \eqref{1st.AP-I4-est} and \eqref{1st.AP-I5-est}, we eventually arrive at 
\begin{align}
\label{1st.AP-final-est}
\begin{split}
\left\|\p_{x}^{l}\psi(\cdot, t)\right\|_{L^{2}}\le&\ CE_{s}(1+t)^{-\frac{3}{4}-\frac{l}{2}} \\
&+C\begin{cases}
(1+t)^{-\frac{\alpha}{2}+\frac{1}{4}-\frac{l}{2}}\|\psi_{0}\|_{L^{\infty}_{\alpha}}, &1\le t\le T, \ 1<\alpha<2,\\
(1+t)^{-\frac{3}{4}-\frac{l}{2}}\log(2+t)\|\psi_{0}\|_{L^{\infty}_{\alpha}}, &1\le t\le T, \ \alpha=2, \\
(1+t)^{-\frac{3}{4}-\frac{l}{2}}\|\psi_{0}\|_{L^{1}_{1}}, &1\le t\le T, \ \alpha>2
\end{cases}\\
&+CE_{s}M_{\alpha}(T)
\begin{cases}
(1+t)^{-\frac{\alpha}{2}+\frac{1}{4}-\frac{l}{2}}, &1\le t\le T, \ 1<\alpha<2,\\
(1+t)^{-\frac{3}{4}-\frac{l}{2}+\e}, &1\le t\le T, \ \alpha\ge2. 
\end{cases}
\end{split}
\end{align}
On the other hand, by using \eqref{u-est-L2}, \eqref{diffusion-decay} and $|M|\le E_{s}$, we find that 
\begin{equation}
\label{1st.AP-t=0-est}
\left\|\p_{x}^{l}\psi(\cdot, t)\right\|_{L^{2}}\le \left\|\p_{x}^{l}u(\cdot, t)\right\|_{L^{2}}+\left\|\p_{x}^{l}\chi(\cdot, t)\right\|_{L^{2}}\le CE_{s}, \ \ 0\le t\le 1.
\end{equation}
Since $\log(2+t)\le C(1+t)^{\e}$, combining \eqref{1st.AP-final-est} and \eqref{1st.AP-t=0-est}, we arrive at 
\begin{equation*}
M_{\alpha}(T)\le C_{0}E_{s}+C_{\alpha}\left(\|\psi_{0}\|_{L^{\infty}_{\alpha}}, \|\psi_{0}\|_{L^{1}_{1}}\right)+C_{1}E_{s}M_{\alpha}(T),
\end{equation*}
where $C_{0}$ and $C_{1}$ are certain positive constants. Therefore, we obtain the desired estimate 
\begin{equation*}
M_{\alpha}(T)\le 2C_{0}E_{s}+2C_{\alpha}\left(\|\psi_{0}\|_{L^{\infty}_{\alpha}}, \|\psi_{0}\|_{L^{1}_{1}}\right)
\end{equation*}
if $E_{s}$ is small that $C_{1}E_{s}\le 1/2$. This completes the proof. 
\end{proof}

\begin{proof}[\rm{\bf{End of the Proof of Theorem~\ref{main.thm-1st.AP}}}]
The estimate \eqref{psi-est-L2} in Proposition~\ref{prop.psi-est-L2} and the first estimate \eqref{main.thm-1st.AP-L2} in Theorem~\ref{main.thm-1st.AP} are equivalent. The second estimate \eqref{main.thm-1st.AP-Linfinity} in Theorem~\ref{main.thm-1st.AP} immediately follows from \eqref{main.thm-1st.AP-L2} through Sobolev's inequality \eqref{Sobolev-ineq}. This completes the proof of Theorem~\ref{main.thm-1st.AP}.
\end{proof}

\section{Second Asymptotic Profiles}  

In this section, we shall prove our second main result Theorem~\ref{main.thm-2nd.AP}, i.e. we construct the second asymptotic profiles for the solutions to \eqref{BBMB}. This section is divided into three subsections below. 

\subsection{Auxiliary Problem}  

In this subsection, we prepare some important Lemmas to analyze the second asymptotic profiles for the solutions to \eqref{BBMB}. To construct the second asymptotic profiles, we need to analyze the Cauchy problems \eqref{difference-second} and \eqref{difference-second-2} below. First, let us prepare the following auxiliary problem:
\begin{align}\label{aux}
\begin{split}
z_{t}+(\beta \chi z)_{x}-z_{xx}&=\lambda_x, \ \ x\in \R, \ t>0, \\
z(x, 0)&=z_{0}(x) , \ \ x\in \R, 
\end{split}
\end{align}
where $\lambda(x, t)$ is a given sufficiently regular function decaying at spatial infinity. If we set 
\begin{align}\label{U}
\begin{split}
U[h](x, t, \tau):=\int_{\R}\p_{x}(G(x-y, t-\tau)\eta(x, t))\eta(y, \tau)^{-1}\left(\int_{-\infty}^{y}h(\xi)d\xi\right)dy&,\\
x\in \R, \ 0\le \tau<t,& 
\end{split}
\end{align}
then, applying Lemma~2.6 in~\cite{F20} or Lemma~5.1 in~\cite{KU17}, we have the following formula. 
\begin{lem}\label{lem.formula}
Let $z_{0}(x)$ be a sufficiently regular function decaying at spatial infinity. Then, we can get the smooth solution of \eqref{aux} which satisfies the following formula: 
\begin{equation}\label{formula}
z(x, t)=U[z_{0}](x, t, 0)+\int_{0}^{t}U[\p_{x}\lambda(\tau)](x, t, \tau)d\tau, \ \ x\in \R, \ t>0.
\end{equation}
\end{lem}
\noindent
This explicit representation formula \eqref{formula} plays an essential role in the proof of Theorem~\ref{main.thm-2nd.AP}. 

Next, for the latter sake, we introduce some estimates for $\eta(x, t)$ defined by
\begin{equation}\tag{\ref{heat-eta}}
\eta(x, t):=\eta_{*}\left(\frac{x}{\sqrt{1+t}}\right)
=\exp \left(\frac{\beta}{2}\int_{-\infty}^{x}\chi(y, t)dy\right).
\end{equation}
First, by a direct calculation, we can easily get the following uniform boundedness: 
\begin{align}
&\min \left\{1, e^{\beta M/2}\right\} \le \eta(x, t) \le \max \left\{1, e^{\beta M/2}\right\}, \ \ x\in \R, \ t\ge0, 
\label{eta-est1}\\ 
&\min \left\{1, e^{-\beta M/2}\right\} \le \eta(x, t)^{-1} \le \max \left\{1, e^{-\beta M/2}\right\}, \ \ x\in \R, \ t\ge0. \label{eta-est2}
\end{align}
Moreover, by using Lemma~\ref{diffusion-decay}, we are able to show the following $L^{p}$-decay estimates for $\eta(x, t)$ and $\eta(x, t)^{-1}$ (for the proof, see Corollary~2.3 in~\cite{K07} or Lemma~5.4 in~\cite{KU17}). 
\begin{lem}\label{lem.eta-decay}
Let $l$ be a positive integer and $1\le p\le \infty$. If $|M| \le1$, then we have 
\begin{align}\label{eta-decay}
\left\| \p^{l}_{x}\eta(\cdot, t)\right\|_{L^{p}}+\left\| \p^{l}_{x}(\eta(\cdot, t)^{-1})\right\|_{L^{p}}&\le C|M|(1+t)^{-\frac{1}{2}(1-\frac{1}{p})-\frac{l}{2}+\frac{1}{2}}, \ \ t\ge0.
\end{align}
\end{lem}

Now, we turn back to \eqref{formula} and assume that the initial data $z_{0}(x)$ satisfies \eqref{data}.
Then, if $1<\alpha\le 2$, we have the following asymptotic formula (Proposition \ref{prop-linear-2nd.AP}) for the first term of \eqref{formula}, which plays an essential role in the proof of Theorem~\ref{main.thm-2nd.AP}. The idea of the proof is originally based on the method used in \cite{NN08}, which studies the damped wave equation and the heat equation (see also \cite{IIOW19, IIW17}). In \cite{F19-2}, the first author modified the original method in \cite{NN08} so that the modified method can be applied to the Burgers type equations. The following result is one of the generalization of Proposition~4.1 in \cite{F19-2}.
\begin{prop}\label{prop-linear-2nd.AP}
Let $l$ be a non-negative integer and $1<\alpha\le 2$. Assume that the initial data $z_{0}(x)$ satisfies the condition \eqref{data} and $\int_{\R}z_{0}(x)dx=0$. In addition, we set 
\begin{equation}\label{data-z0-2}
r_{0}(x):=\eta_{*}(x)^{-1}\int_{-\infty}^{x}z_{0}(y)dy, \quad 
\eta_{*}(x):=\exp \left(\frac{\beta}{2}\int_{-\infty}^{x}\chi_{*}(y)dy\right),
\end{equation}
and suppose that there exist $\lim_{x\to \pm \infty}(1+|x|)^{\alpha-1}r_{0}(x)=:c_{\alpha}^{\pm}$. Then, we have 
\begin{align}
&\lim_{t\to \infty}(1+t)^{\frac{\alpha}{2}+\frac{l}{2}}\left\|\p_{x}^{l}(U[z_{0}](\cdot, t, 0)-Z(\cdot, t))\right\|_{L^{\infty}}=0, \ \ 1<\alpha<2, \label{linear-2nd.AP-1}\\
&\lim_{t\to \infty}\frac{(1+t)^{1+\frac{l}{2}}}{\log(1+t)}\left\|\p_{x}^{l}(U[z_{0}](\cdot, t, 0)-Z(\cdot, t))\right\|_{L^{\infty}}=0, \ \ \alpha=2,  \label{linear-2nd.AP-2}
\end{align}
where $U$ and $Z(x, t)$ are defined by \eqref{U} and \eqref{Second-AP-Z}, respectively. 
\end{prop}
\begin{proof}
It follows from the definition of $U$ given by \eqref{U} that 
\begin{align}\label{U-t}
\begin{split}
U[z_{0}](x, t, 0)=&\int_{\R}\p_{x}(G(x-y, t)\eta(x, t))\eta_{*}(y)^{-1}\left(\int_{-\infty}^{y}z_{0}(\xi)d\xi\right) dy\\
=&\int_{\R}\p_{x}(G(x-y, t)\eta(x, t))r_{0}(y)dy.
\end{split}
\end{align} 
First, we shall check the following estimate: 
\begin{equation}\label{r0-decay}
|r_{0}(y)|\le C(1+|y|)^{-(\alpha-1)}, \ y\in \R.
\end{equation}
If $x<0$, from \eqref{eta-est2} and \eqref{data}, we obtain 
\begin{align*}
|r_{0}(x)|&\le C\int_{-\infty}^{x}|z_{0}(y)|dy
\le C\int_{-\infty}^{x}(1+|y|)^{-\alpha}dy \\
&\le C\int_{-\infty}^{x}(1-y)^{-\alpha}dy\le C(1-x)^{-(\alpha-1)}=C(1+|x|)^{-(\alpha-1)}. 
\end{align*}
On the other hand, since $\int_{\R}z_{0}(x)dx=0$, if $x>0$, similarly we have
\begin{align*}
|r_{0}(x)|&\le C\left|\int_{-\infty}^{x}z_{0}(y)dy\right|
=C\left|\int_{-\infty}^{x}z_{0}(y)dy-\int_{\R}z_{0}(y)dy\right|
=C\left|\int_{x}^{\infty}z_{0}(y)dy\right|\\
&\le C\int_{x}^{\infty}(1+|y|)^{-\alpha}dy
\le C\int_{x}^{\infty}(1+y)^{-\alpha}dy\le C(1+x)^{-(\alpha-1)}=C(1+|x|)^{-(\alpha-1)}. 
\end{align*}
Therefore, we get \eqref{r0-decay}. Thus, we obtain the boundedness of $(1+|y|)^{\alpha-1}r_{0}(y)$. Moreover, from the assumption on $r_{0}(y)$, for any $\e>0$ there is a constant $R=R(\e)>0$ such that
 \begin{align*}
&\left|r_{0}(y)-c_{\alpha}^{+}(1+|y|)^{-(\alpha-1)}\right|\le \e(1+|y|)^{-(\alpha-1)}, \ y\ge R, \\
&\left|r_{0}(y)-c_{\alpha}^{-}(1+|y|)^{-(\alpha-1)}\right|\le \e(1+|y|)^{-(\alpha-1)}, \ y\le -R.
 \end{align*}
From \eqref{Second-AP-Z} and \eqref{U-t}, we obtain the following estimate
 \begin{align}\label{U-Z-est1}
 \begin{split}
&\left|\p_{x}^{l}(U[z_{0}](x, t, 0)-Z(x, t))\right|\\
 &\le \int_{\R}\left|\p_{x}^{l+1}(G(x-y, t)\eta(x, t))\right|\left|r_{0}(y)-c_{\alpha}(y)(1+|y|)^{-(\alpha-1)}\right|dy\\
 &\le \int_{|y|\le R}\left|\p_{x}^{l+1}(G(x-y, t)\eta(x, t))\right|\left|r_{0}(y)-c_{\alpha}(y)(1+|y|)^{-(\alpha-1)}\right|dy\\
 &\ \ \ \ +\e \int_{|y|\ge R}\left|\p_{x}^{l+1}(G(x-y, t)\eta(x, t))\right|(1+|y|)^{-(\alpha-1)}dy \\
 &\le C\sum_{n=0}^{l+1}\left\|\p_{x}^{l+1-n}\eta(\cdot, t)\right\|_{L^{\infty}}\left\|\p_{x}^{n}G(\cdot, t)\right\|_{L^{\infty}}\int_{|y|\le R}\left|r_{0}(y)-c_{\alpha}(y)(1+|y|)^{-(\alpha-1)}\right|dy\\
 &\ \ \ +\e C\sum_{n=0}^{l+1}\left\|\p_{x}^{l+1-n}\eta(\cdot, t)\right\|_{L^{\infty}}\int_{\R}\left|\p_{x}^{n}G(x-y, t)\right|(1+|y|)^{-(\alpha-1)}dy.
 \end{split}
 \end{align}
For the integral in the second term of the right hand side of \eqref{U-Z-est1}, we have from \eqref{heat-decay} that 
 \begin{align}
 \label{G-z0-int}
\begin{split}
&\int_{\R} \left|\p_{x}^{n}G(x-y, t)\right|(1+|y|)^{-(\alpha-1)}dy \\
&=\left(\int_{|y|\ge \sqrt{1+t}-1}+\int_{|y|\le \sqrt{1+t}-1}\right)\left|\p_{x}^{n}G(x-y, t)\right|(1+|y|)^{-(\alpha-1)}dy \\
&\le \left(\sup_{|y|\ge \sqrt{1+t}-1}(1+|y|)^{-(\alpha-1)}\right)\int_{|y|\ge \sqrt{1+t}-1}\left|\p_{x}^{n}G(x-y, t)\right|dy\\
&\ \ \ \ +\left(\sup_{|y|\le \sqrt{1+t}-1}\left|\p_{x}^{n}G(x-y, t)\right|\right)\int_{|y|\le \sqrt{1+t}-1}(1+|y|)^{-(\alpha-1)}dy\\
&\le (1+t)^{-\frac{\alpha-1}{2}}\left\|\p_{x}^{n}G(\cdot, t)\right\|_{L^{1}}+\left\|\p_{x}^{n}G(\cdot, t)\right\|_{L^{\infty}}\int_{|y|\le \sqrt{1+t}-1}(1+|y|)^{-(\alpha-1)}dy\\
&\le C(1+t)^{-\frac{\alpha-1}{2}-\frac{n}{2}}+Ct^{-\frac{1}{2}-\frac{n}{2}}\int_{0}^{\sqrt{1+t}-1}(1+y)^{-(\alpha-1)}dy \\
&\le C(1+t)^{-\frac{\alpha-1}{2}-\frac{n}{2}}+Ct^{-\frac{1}{2}-\frac{n}{2}}\begin{cases}
(1+t)^{-\frac{\alpha-1}{2}+\frac{1}{2}}, &1<\alpha<2,  \\
\log(2+t), &\alpha=2
\end{cases}\\
&\le C\begin{cases}
(1+t)^{-\frac{\alpha-1}{2}-\frac{n}{2}}, &t\ge1, \ 1<\alpha<2, \\
(1+t)^{-\frac{1}{2}-\frac{n}{2}}\log(2+t), &t\ge1, \ \alpha=2.
\end{cases}
\end{split}
\end{align} 
Therefore, by using \eqref{U-Z-est1}, \eqref{G-z0-int}, \eqref{heat-decay} and \eqref{eta-decay}, we obtain
  \begin{equation*}
 \left\|\p_{x}^{l}(U[z_{0}](\cdot, t, 0)-Z(\cdot, t))\right\|_{L^{\infty}}\le C(1+t)^{-1-\frac{l}{2}}+\e C\begin{cases}
(1+t)^{-\frac{\alpha}{2}-\frac{l}{2}}, &t\ge1, \ 1<\alpha<2, \\
(1+t)^{-1-\frac{l}{2}}\log(2+t), &t\ge1, \ \alpha=2.\end{cases}
 \end{equation*}
Thus, we finally arrive at 
 \begin{align*}
&\limsup_{t\to \infty}(1+t)^{\frac{\alpha}{2}+\frac{l}{2}}\left\|\p_{x}^{l}(U[z_{0}](\cdot, t, 0)-Z(\cdot, t))\right\|_{L^{\infty}}\le\e C, \ \ 1<\alpha<2, \\
&\limsup_{t\to \infty}\frac{(1+t)^{1+\frac{l}{2}}}{\log(1+t)}\left\|\p_{x}^{l}(U[z_{0}](\cdot, t, 0)-Z(\cdot, t))\right\|_{L^{\infty}} \le\e C, \ \ \alpha=2.  
\end{align*}
Therefore, we get \eqref{linear-2nd.AP-1} and \eqref{linear-2nd.AP-2}, because $\e>0$ can be chosen arbitrarily small. 
\end{proof}

In addition to the above proposition, we have the decay estimates for the first term of \eqref{formula}. Actually, the following estimates have been established (for the proof, see Corollary~3.4 in~\cite{K07}). 
\begin{lem}\label{lem.est-U1}
Let $s\ge1$ be an integer and $\alpha>2$. Assume that the initial data $z_{0}(x)$ satisfies the condition \eqref{data}, $z_{0}\in H^{s}(\R)$ and $\int_{\R}z_{0}(x)dx=0$. If $|M| \le1$, then the estimate 
\begin{equation}\label{est-U1-L2}
\left\| \p^{l}_{x}U[z_{0}](\cdot, t, 0)\right\|_{L^{2}} \le C(\|z_{0}\|_{H^{s}}+\|z_{0}\|_{L^{1}_{1}})(1+t)^{-\frac{3}{4}-\frac{l}{2}}, \ \ t>0
\end{equation}
holds for any integer $l$ satisfying $0\le l \le s$. Moreover, 
\begin{equation}\label{est-U1-Linfinity}
\left\| \p^{l}_{x}U[z_{0}](\cdot, t, 0)\right\|_{L^{\infty}} \le C(\|z_{0}\|_{H^{s}}+\|z_{0}\|_{L^{1}_{1}})(1+t)^{-1-\frac{l}{2}}, \ \ t>0
\end{equation}
holds for any integer $l$ satisfying $0\le l \le s-1$. Here, $U$ is defined by \eqref{U}. 
\end{lem}

Finally, to evaluate the second term of \eqref{formula}, we prepare the following lemma: 
\begin{lem}\label{lem.est-U2}
Let $l$ be a non-negative integer and $1\le p\le \infty$. Assume that $\lambda(x, t)$ is sufficiently regular function decaying at spatial infinity. If $|M| \le1$, then we have 
\begin{equation}\label{est-U2}
\left\|\p_{x}^{l}U[\p_{x}\lambda(\tau)](\cdot, t, \tau)\right\|_{L^{p}}\le C\sum_{j=0}^{l+1}(1+t)^{-\frac{1}{2}(l+1-j)}\left\|\p_{x}^{j}J[\lambda](\cdot, t, \tau)\right\|_{L^{p}}, 
\end{equation}
where $U$ is defined by \eqref{U} and $J$ is defined by 
\begin{equation}\label{J}
J[\lambda](x, t, \tau):=\int_{\R}G(x-y, t-\tau)\eta(y, \tau)^{-1}\lambda(y, \tau)dy=\left(G(t-\tau)*(\eta^{-1}\lambda \right)(\tau))(x).
\end{equation}
\end{lem}
\begin{proof}
For any given regular function $\lambda(x, t)$ and any integer $l$, it follows from \eqref{U} that 
\begin{align*}
\p_{x}^{l}U[\p_{x}\lambda(\tau)](x, t, \tau)=\sum_{j=0}^{l+1}\begin{pmatrix}l+1 \\ j \end{pmatrix}\p_{x}^{l+1-j}\eta(x, t)\int_{\R}\p_{x}^{j}G(x-y, t-\tau)\eta(y, \tau)^{-1}\lambda(y, \tau)dy. 
\end{align*}
Therefore, the desired result \eqref{est-U2} immediately follows from Lemma~\ref{lem.eta-decay}, \eqref{eta-est1} and \eqref{J}. 
\end{proof}

\subsection{Proof of Theorem \ref{main.thm-2nd.AP}~for~$\bm{1<\alpha<2}$}  

In this subsection, we shall prove Theorem~\ref{main.thm-2nd.AP} in the case of $1<\alpha<2$. Namely, we show that the asymptotic relation \eqref{main.thm-2nd.AP-1} holds. First, multiplying the operator $(1-\p_{x}^{2})^{-1}$ on \eqref{BBMB}, we have 
\begin{align}
\label{BBMB-up}
\begin{split}
&u_{t}-u_{xxt}-u_{xx}+\gamma u_{xxx}+\beta uu_{x}=0\\
\Longleftrightarrow \ \ &u_{t}-(1-\p_{x}^{2})^{-1}\p_{x}^{2}u+\gamma (1-\p_{x}^{2})^{-1}\p_{x}^{3}u+\beta (1-\p_{x}^{2})^{-1}(uu_{x})=0 \\
\Longleftrightarrow \ \ &u_{t}-u_{xx}-(1-\p_{x}^{2})^{-1}\p_{x}^{4}u+\gamma (1-\p_{x}^{2})^{-1}\p_{x}^{3}u+\beta uu_{x}+\beta (1-\p_{x}^{2})^{-1}\p_{x}^{2}(uu_{x})=0 \\
\Longleftrightarrow \ \ &u_{t}+\beta uu_{x}-u_{xx}=-\gamma (1-\p_{x}^{2})^{-1}\p_{x}^{3}u+(1-\p_{x}^{2})^{-1}\p_{x}^{4}u-\frac{\beta}{2} (1-\p_{x}^{2})^{-1}\p_{x}^{3}(u^{2}), 
\end{split}
\end{align}
where we have used the following facts: 
\begin{align*}
&(1-\p_{x}^{2})^{-1}\p_{x}^{2}u=u_{xx}+(1-\p_{x}^{2})^{-1}\p_{x}^{4}u, \\
&(1-\p_{x}^{2})^{-1}(uu_{x})=uu_{x}+(1-\p_{x}^{2})^{-1}\p_{x}^{2}(uu_{x})
=u_{xx}+\frac{1}{2}(1-\p_{x}^{2})^{-1}\p_{x}^{3}(u^{2}).
\end{align*}
Now, recalling the definitions of $\psi(x, t)$ and $\psi_{0}(x)$:
\begin{equation}\tag{\ref{difference}}
\psi(x, t):=u(x, t)-\chi(x, t), \quad \psi_{0}(x):=u_{0}(x)-\chi_{*}(x). 
\end{equation}
Then, from \eqref{BBMB-up} and \eqref{Burgers}, we have the following Cauchy problem: 
\begin{align}\label{difference-second}
\begin{split}
\psi_{t}+(\beta \chi \psi)_{x}-\psi_{xx}
&=-\p_{x}\left(\frac{\beta}{2}\psi^{2}\right)
-\gamma(1-\p_{x}^{2})^{-1}\p_{x}^{3}u\\
&\ \ \ \,+(1-\p_{x}^{2})^{-1}\p_{x}^{4}u
-\frac{\beta}{2}(1-\p_{x}^{2})^{-1}\p_{x}^{3}(u^{2}), \ \ x\in \R, \ t>0,\\
\psi(x, 0)&=u_{0}(x)-\chi_{*}(x)=\psi_{0}(x), \ \ x\in \R.
\end{split}
\end{align}
Applying Lemma~\ref{lem.formula} to \eqref{difference-second}, we obtain
\begin{equation}\label{integral-eq-2nd}
\psi(x, t)=U[\psi_{0}](x, t, 0)+D[u, \psi](x, t), 
\end{equation}
where the integral part $D[u, \psi](x, t)$ is defined by
\begin{align}\label{integral-eq-D}
\begin{split}
&D[u, \psi](x, t):=-\frac{\beta}{2}\int_{0}^{t}U[\p_{x}(\psi^{2})(\tau)](x, t, \tau)d\tau
-\gamma \int_{0}^{t}U[(1-\p_{x}^{2})^{-1}\p_{x}^{3}u(\tau)](x, t, \tau)d\tau \\
&\quad +\int_{0}^{t}U[(1-\p_{x}^{2})^{-1}\p_{x}^{4}u(\tau)](x, t, \tau)d\tau
-\frac{\beta}{2}\int_{0}^{t}U[(1-\p_{x}^{2})^{-1}\p_{x}^{3}(u^{2})(\tau)](x, t, \tau)d\tau \\
&\quad =:D_{1}(x, t)+D_{2}(x, t)+D_{3}(x, t)+D_{4}(x, t). 
\end{split}
\end{align}

In order to prove \eqref{main.thm-2nd.AP-1}, it is sufficient to show the following proposition: 
\begin{prop}\label{prop.D-est}
Let $s\ge2$ be an integer and $1<\alpha<2$. Assume that the initial data $u_{0}(x)$ satisfies the condition \eqref{data}, $u_{0}\in H^{s}(\R)$ and $\|u_{0}\|_{H^{s}}+\|u_{0}\|_{L^{1}}$ is sufficiently small. Then, the estimate  
\begin{equation}\label{D-est}
\left\|\p_{x}^{l}D[u, \psi](\cdot, t)\right\|_{L^{\infty}}\le C
\begin{cases}
(1+t)^{-\alpha+\frac{1}{2}-\frac{l}{2}}, &t\ge1, \ 1<\alpha<3/2, \\
(1+t)^{-1-\frac{l}{2}}\log(2+t), &t\ge1, \ 3/2 \le \alpha<2
\end{cases}
\end{equation}
holds for any integer $l$ satisfying $0\le l\le s-2$.
\end{prop}
\begin{proof}
We shall evaluate each terms $D_{i}(x, t)$ ($i=1, 2, 3, 4$) in the right hand side of \eqref{integral-eq-D}. 
First, we start with evaluation of $D_{1}(x, t)$. Applying Lemma~\ref{lem.est-U2} and splitting the $\tau$-integral, we have 
\begin{align}
\begin{split}\label{D1-split}
\left\|\p_{x}^{l}D_{1}(\cdot, t)\right\|_{L^{\infty}}
&\le C\sum_{j=0}^{l+1}(1+t)^{-\frac{l+1-j}{2}}\int_{0}^{t}\left\|\p_{x}^{j}J[\psi^{2}](\cdot, t, \tau)\right\|_{L^{\infty}}d\tau\\
&\le C\sum_{j=0}^{l+1}(1+t)^{-\frac{l+1-j}{2}}\left(\int_{0}^{t/2}+\int_{t/2}^{t}\right)\left\|\p_{x}^{j}J[\psi^{2}](\cdot, t, \tau)\right\|_{L^{\infty}}d\tau\\
&=:C\sum_{j=0}^{l+1}(1+t)^{-\frac{l+1-j}{2}}(D_{1.1}+D_{1.2}).
\end{split}
\end{align}
For $D_{1.1}$, from \eqref{J}, Young's inequality, \eqref{eta-est2}, \eqref{heat-decay} and \eqref{main.thm-1st.AP-L2}, we obtain 
\begin{align}
\begin{split}\label{D1.1-est}
D_{1.1}&\le C\int_{0}^{t/2}\left\|\p_{x}^{j}G(\cdot, t-\tau)\right\|_{L^{\infty}}\left\|\psi^{2}(\cdot, \tau)\right\|_{L^{1}}d\tau
\le C\int_{0}^{t/2}(t-\tau)^{-\frac{1}{2}-\frac{j}{2}}(1+\tau)^{-\alpha+\frac{1}{2}}d\tau  \\
&\le C(1+t)^{-\frac{1}{2}-\frac{j}{2}}
\begin{cases}
(1+t)^{-\alpha+\frac{3}{2}}, &t\ge1, \ 1<\alpha<3/2, \\
\log(2+t), &t\ge1, \ \alpha=3/2, \\
C,  &t\ge1, \ 3/2<\alpha<2. 
\end{cases}
\end{split}
\end{align}
On the other hand, for $D_{1.2}$ with $j=0$, we have from \eqref{heat-decay} and \eqref{main.thm-1st.AP-Linfinity} that 
\begin{align}
\begin{split}\label{D1.2-est-1}
D_{1.2}&\le C\int_{t/2}^{t}\|G(\cdot, t-\tau)\|_{L^{1}}\left\|\psi^{2}(\cdot, \tau)\right\|_{L^{\infty}}d\tau 
\le C\int_{t/2}^{t}\|\psi(\cdot, \tau)\|_{L^{\infty}}^{2}d\tau\\
&\le C\int_{t/2}^{t}(1+\tau)^{-\alpha}d\tau \le C(1+t)^{-\alpha+1}, \ \ t\ge0. 
\end{split}
\end{align}
For $j=1, 2, \cdots, l+1$, we prepare an estimate. It follows from \eqref{lem.eta-decay}, \eqref{eta-est2} and \eqref{main.thm-1st.AP-Linfinity} that 
\begin{align}\label{prepare-for-D1}
\begin{split}
\left\|\p_{x}^{j-1}((\eta^{-1}\psi^{2})(\cdot, \tau))\right\|_{L^{\infty}}
& \le C\sum_{m=0}^{j-1}\sum_{n=0}^{j-1-m}\left\|\p_{x}^{m}(\eta(\cdot, \tau)^{-1})\right\|_{L^{\infty}}\left\|\p_{x}^{n}\psi(\cdot, \tau)\right\|_{L^{\infty}}\left\|\p_{x}^{j-1-m-n}\psi(\cdot, \tau)\right\|_{L^{\infty}}\\
&\le C\sum_{m=0}^{j-1}\sum_{n=0}^{j-1-m}(1+\tau)^{-\frac{m}{2}}(1+\tau)^{-\frac{\alpha}{2}-\frac{n}{2}}(1+\tau)^{-\frac{\alpha}{2}-\frac{j-1-m-n}{2}}\\
&\le C(1+\tau)^{-\alpha-\frac{j-1}{2}}, \ \ t\ge0.
\end{split}
\end{align}
Therefore, we get the following estimate for $D_{1.2}$ with $j=1, 2, \cdots, l+1$: 
\begin{align}
\begin{split}\label{D1.2-est-2}
D_{1.2}&\le C\int_{t/2}^{t}\left\|\p_{x}G(\cdot, t-\tau)\right\|_{L^{1}}\left\|\p_{x}^{j-1}((\eta^{-1}\psi^{2})(\cdot, \tau))\right\|_{L^{\infty}}d\tau\\
&\le C\int_{t/2}^{t}(t-\tau)^{-\frac{1}{2}}(1+\tau)^{-\alpha-\frac{j-1}{2}}d\tau \le C(1+t)^{-\alpha+1-\frac{j}{2}}, \ \ t\ge0. 
\end{split}
\end{align}
Combining \eqref{D1.2-est-1} and \eqref{D1.2-est-2}, for all $j=0, 1, \cdots, l+1$, we have
\begin{equation}\label{D1.2-est}
D_{1.2}\le C(1+t)^{-\alpha+1-\frac{j}{2}}, \ \ t\ge0.
\end{equation}
Therefore, summing up \eqref{D1-split}, \eqref{D1.1-est} and \eqref{D1.2-est}, we arrive at 
\begin{align}\label{D1-est}
\begin{split}
\left\|\p_{x}^{l}D_{1}(\cdot, t)\right\|_{L^{\infty}}
&\le C(1+t)^{-\alpha+\frac{1}{2}-\frac{l}{2}}
+C(1+t)^{-1-\frac{l}{2}}
\begin{cases}
(1+t)^{-\alpha+\frac{3}{2}}, &1<\alpha<3/2, \\
\log(2+t), &\alpha=3/2, \\
C,  &3/2<\alpha<2
\end{cases}\\
&\le C\begin{cases}
(1+t)^{-\alpha+\frac{1}{2}-\frac{l}{2}}, &t\ge1, \ 1<\alpha<3/2, \\
(1+t)^{-1-\frac{l}{2}}\log(2+t), &t\ge1, \ \alpha=3/2, \\
(1+t)^{-1-\frac{l}{2}},  &t\ge1, \ 3/2<\alpha<2,
\end{cases}\quad 0\le l \le s-1.
\end{split}
\end{align}

Next, we deal with the estimate for $D_{2}(x, t)$. In the following, let $0\le l \le s-2$. Similarly as \eqref{D1-split}, applying Lemma~\ref{lem.est-U2} and splitting the $\tau$-integral, we have 
\begin{align}
\begin{split}\label{D2-split}
&\left\|\p_{x}^{l}D_{2}(\cdot, t)\right\|_{L^{\infty}}
\le C\sum_{j=0}^{l+1}(1+t)^{-\frac{l+1-j}{2}}\int_{0}^{t}\left\|\p_{x}^{j}J[(1-\p_{x}^{2})^{-1}\p_{x}^{2}u](\cdot, t, \tau)\right\|_{L^{\infty}}d\tau\\
&\le C\sum_{j=0}^{l+1}(1+t)^{-\frac{l+1-j}{2}}\left(\int_{0}^{t/2}+\int_{t/2}^{t}\right)\left\|\p_{x}^{j}J[(1-\p_{x}^{2})^{-1}\p_{x}^{2}u](\cdot, t, \tau)\right\|_{L^{\infty}}d\tau\\
&=:C\sum_{j=0}^{l+1}(1+t)^{-\frac{l+1-j}{2}}(D_{2.1}+D_{2.2}).
\end{split}
\end{align}
First, making the integration by parts, we rewrite $J[(1-\p_{x}^{2})^{-1}\p_{x}^{2}u](x, t, \tau)$ as follows 
\begin{align}
\begin{split}\label{int-parts}
J[(1-\p_{x}^{2})^{-1}\p_{x}^{2}u](x, t, \tau)&
=\int_{\R}G(x-y, t-\tau)\eta(y, \tau)^{-1}(1-\p_{y}^{2})^{-1}\p_{y}^{2}u(y, \tau)dy\\
&=\int_{\R}\p_{x}G(x-y, t-\tau)\eta(y, \tau)^{-1}(1-\p_{y}^{2})^{-1}\p_{y}u(y, \tau)dy\\
&\ \ \ \ -\int_{\R}G(x-y, t-\tau)\p_{y}(\eta(y, \tau)^{-1})(1-\p_{y}^{2})^{-1}\p_{y}u(y, \tau)dy.
\end{split}
\end{align}
Therefore, by using \eqref{int-parts}, Young's inequality, \eqref{eta-est2}, \eqref{heat-decay}, Lemma~\ref{lem.embedding}, Schwarz's inequality, \eqref{u-est-L2} and Lemma~\ref{lem.eta-decay}, we can see that 
\begin{align}
\begin{split}\label{D2.1-est}
&D_{2.1}\le C\int_{0}^{t/2}\left\|\p_{x}^{j+1}G(\cdot, t-\tau)\right\|_{L^{2}}\left\|(1-\p_{x}^{2})^{-1}\p_{x}u(\cdot, \tau)\right\|_{L^{2}}d\tau \\
&\qquad \quad+C\int_{0}^{t/2}\left\|\p_{x}^{j}G(\cdot, t-\tau)\right\|_{L^{\infty}}\left\|\p_{x}(\eta(\cdot, \tau)^{-1})(1-\p_{x}^{2})^{-1}\p_{x}u(\cdot, \tau)\right\|_{L^{1}}d\tau \\
&\le C\int_{0}^{t/2}\left\|\p_{x}^{j+1}G(\cdot, t-\tau)\right\|_{L^{2}}\left\|\p_{x}u(\cdot, \tau)\right\|_{L^{2}}d\tau \\
&\ \ \ +C\int_{0}^{t/2}\left\|\p_{x}^{j}G(\cdot, t-\tau)\right\|_{L^{\infty}}\left\|\p_{x}(\eta(\cdot, \tau)^{-1})\right\|_{L^{2}}\left\|\p_{x}u(\cdot, \tau)\right\|_{L^{2}}d\tau \\
&\le CE_{s}\int_{0}^{t/2}(t-\tau)^{-\frac{3}{4}-\frac{j}{2}}(1+\tau)^{-\frac{3}{4}}d\tau
+CE_{s}\int_{0}^{t/2}(t-\tau)^{-\frac{1}{2}-\frac{j}{2}}(1+\tau)^{-\frac{1}{4}}(1+\tau)^{-\frac{3}{4}}d\tau\\
&\le CE_{s}(1+t)^{-\frac{1}{2}-\frac{j}{2}}+CE_{s}(1+t)^{-\frac{1}{2}-\frac{j}{2}}\log(2+t), \ \ t\ge1.
\end{split}
\end{align}
For $i=0, 1$, by using \eqref{eta-est2}, Lemma~\ref{lem.eta-decay}, Lemma~\ref{lem.embedding} and \eqref{u-est-L2}, we get
\begin{align}\label{prepare-D2.2}
\begin{split}
&\left\|\p_{x}^{j}\left(\p_{x}^{i}(\eta(\cdot, \tau)^{-1})(1-\p_{x}^{2})^{-1}\p_{x}u(\cdot, \tau)\right)\right\|_{L^{2}} \\
&\le C\sum_{n=0}^{j}\left\|\p_{x}^{j+i-n}(\eta(\cdot, \tau)^{-1})\right\|_{L^{\infty}}\left\|(1-\p_{x}^{2})^{-1}\p_{x}^{n+1}u(\cdot, \tau)\right\|_{L^{2}}  \\
&\le C\sum_{n=0}^{j}\left\|\p_{x}^{j+i-n}(\eta(\cdot, \tau)^{-1})\right\|_{L^{\infty}}\left\|\p_{x}^{n+1}u(\cdot, \tau)\right\|_{L^{2}}\\
&\le CE_{s}\sum_{n=0}^{j}(1+\tau)^{-\frac{j+i-n}{2}}(1+\tau)^{-\frac{3}{4}-\frac{n}{2}}
\le CE_{s}(1+\tau)^{-\frac{3}{4}-\frac{i+j}{2}}, \ \ t\ge0.
\end{split}
\end{align}
Thus, from \eqref{int-parts}, Young's inequality, \eqref{heat-decay} and \eqref{prepare-D2.2}, we have
\begin{align}
\begin{split}\label{D2.2-est}
D_{2.2}&\le C\int_{t/2}^{t}\left\|\p_{x}G(\cdot, t-\tau)\right\|_{L^{2}}\left\|\p_{x}^{j}\left(\eta(\cdot, \tau)^{-1}(1-\p_{x}^{2})^{-1}\p_{x}u(\cdot, \tau)\right)\right\|_{L^{2}}d\tau \\
&\ \ \ +C\int_{t/2}^{t}\|G(\cdot, t-\tau)\|_{L^{2}}\left\|\p_{x}^{j}\left(\p_{x}(\eta(\cdot, \tau)^{-1})(1-\p_{x}^{2})^{-1}\p_{x}u(\cdot, \tau)\right)\right\|_{L^{2}}d\tau \\
&\le CE_{s}\int_{t/2}^{t}(t-\tau)^{-\frac{3}{4}}(1+\tau)^{-\frac{3}{4}-\frac{j}{2}}d\tau
+CE_{s}\int_{t/2}^{t}(t-\tau)^{-\frac{1}{4}}(1+\tau)^{-\frac{5}{4}-\frac{j}{2}}d\tau\\
&\le CE_{s}(1+t)^{-\frac{1}{2}-\frac{j}{2}}, \ \ t\ge0.
\end{split}
\end{align}
Summing up \eqref{D2-split}, \eqref{D2.1-est} and \eqref{D2.2-est}, we get 
\begin{equation}\label{D2-est}
\left\|\p_{x}^{l}D_{2}(\cdot, t)\right\|_{L^{\infty}}\le CE_{s}(1+t)^{-1-\frac{l}{2}}\log(2+t
), \ \ t\ge1, \ \ 0\le l\le s-2.
\end{equation}

Next, let us treat the evaluation for $D_{3}(x, t)$. From Lemma~\ref{lem.est-U2}, we can see that 
\begin{align}
\begin{split}\label{D3-def}
&\left\|\p_{x}^{l}D_{3}(\cdot, t)\right\|_{L^{\infty}}
\le C\sum_{j=0}^{l+1}(1+t)^{-\frac{l+1-j}{2}}\int_{0}^{t}\left\|\p_{x}^{j}J[(1-\p_{x}^{2})^{-1}\p_{x}^{3}u](\cdot, t, \tau)\right\|_{L^{\infty}}d\tau\\
&\le C\sum_{j=0}^{l+1}(1+t)^{-\frac{l+1-j}{2}}\left(\int_{0}^{t/2}+\int_{t/2}^{t}\right)\left\|\p_{x}^{j}J[(1-\p_{x}^{2})^{-1}\p_{x}^{3}u](\cdot, t, \tau)\right\|_{L^{\infty}}d\tau\\
&=:C\sum_{j=0}^{l+1}(1+t)^{-\frac{l+1-j}{2}}(D_{3.1}+D_{3.2}).
\end{split}
\end{align}
Also, in the same way to get \eqref{int-parts}, making the integration by parts, we obtain 
\begin{align}
\begin{split}\label{int-parts-2}
J[(1-\p_{x}^{2})^{-1}\p_{x}^{3}u](x, t, \tau)
&=\int_{\R}G(x-y, t-\tau)\eta(y, \tau)^{-1}(1-\p_{y}^{2})^{-1}\p_{y}^{3}u(y, \tau)dy \\
&=\int_{\R}\p_{x}G(x-y, t-\tau)\eta(y, \tau)^{-1}(1-\p_{y}^{2})^{-1}\p_{y}^{2}u(y, \tau)dy\\
&\ \ \ \ -\int_{\R}G(x-y, t-\tau)\p_{y}(\eta(y, \tau)^{-1})(1-\p_{y}^{2})^{-1}\p_{y}^{2}u(y, \tau)dy.
\end{split}
\end{align}
Thus, completely same as \eqref{D2.1-est}, we can see from \eqref{int-parts-2} that 
\begin{align}
\begin{split}\label{D3.1-est}
&D_{3.1}\le C\int_{0}^{t/2}\left\|\p_{x}^{j+1}G(\cdot, t-\tau)\right\|_{L^{2}}\left\|(1-\p_{x}^{2})^{-1}\p_{x}^{2}u(\cdot, \tau)\right\|_{L^{2}}d\tau \\
&\qquad \quad+C\int_{0}^{t/2}\left\|\p_{x}^{j}G(\cdot, t-\tau)\right\|_{L^{\infty}}\left\|\p_{x}(\eta(\cdot, \tau)^{-1})(1-\p_{x}^{2})^{-1}\p_{x}^{2}u(\cdot, \tau)\right\|_{L^{1}}d\tau \\
&\le C\int_{0}^{t/2}\left\|\p_{x}^{j+1}G(\cdot, t-\tau)\right\|_{L^{2}}\left\|\p_{x}^{2}u(\cdot, \tau)\right\|_{L^{2}}d\tau \\
&\ \ \ +C\int_{0}^{t/2}\left\|\p_{x}^{j}G(\cdot, t-\tau)\right\|_{L^{\infty}}\left\|\p_{x}(\eta(\cdot, \tau)^{-1})\right\|_{L^{2}}\left\|\p_{x}^{2}u(\cdot, \tau)\right\|_{L^{2}}d\tau \\
&\le CE_{s}\int_{0}^{t/2}(t-\tau)^{-\frac{3}{4}-\frac{j}{2}}(1+\tau)^{-\frac{5}{4}}d\tau
+CE_{s}\int_{0}^{t/2}(t-\tau)^{-\frac{1}{2}-\frac{j}{2}}(1+\tau)^{-\frac{1}{4}}(1+\tau)^{-\frac{5}{4}}d\tau\\
&\le CE_{s}(1+t)^{-\frac{1}{2}-\frac{j}{2}}, \ \ t\ge1.
\end{split}
\end{align}
We can omit the evaluation of $D_{3.2}$. Actually, from Plancherel's theorem and \eqref{nonlocal-op}, we obtain 
\begin{align}\label{embedding-2}
\begin{split}
\left\|(1-\p_{x}^{2})^{-1}\p_{x}^{j+2}u(\cdot, t)\right\|_{L^{2}}
&=\left\|\frac{(i\xi)^{j+2}}{1+\xi^{2}}\widehat{u}(\xi, t)\right\|_{L^{2}_{\xi}}
=\left\|\frac{|\xi|}{1+\xi^{2}}(i\xi)^{j+1}\widehat{u}(\xi, t)\right\|_{L^{2}_{\xi}}  \\
&\le \left\|(i\xi)^{j+1}\widehat{u}(\xi, t)\right\|_{L^{2}_{\xi}}
=\left\|\p_{x}^{j+1}u(\cdot, t)\right\|_{L^{2}}. 
\end{split}
\end{align}
Therefore, by virtue of \eqref{embedding-2}, we can evaluate $D_{3.2}$ in essentially the same way to get \eqref{D2.2-est}. Indeed, it follows from \eqref{D3-def} and \eqref{int-parts-2} that the following estimate is true:  
\begin{equation}\label{D3.2-est}
D_{3.2}\le CE_{s}(1+t)^{-\frac{1}{2}-\frac{j}{2}}, \ \ t\ge1. 
\end{equation}
Thus, combining \eqref{D3-def}, \eqref{D3.1-est} and \eqref{D3.2-est}, we obtain 
\begin{equation}\label{D3-est}
\left\|\p_{x}^{l}D_{3}(\cdot, t)\right\|_{L^{\infty}}\le CE_{s}(1+t)^{-1-\frac{l}{2}}, \ \ t\ge1, \ \ 0\le l\le s-2.
\end{equation}

Finally, we shall evaluate $D_{4}(x, t)$. Similarly as before, it follows from Lemma~\ref{lem.est-U2} that 
\begin{align}
\begin{split}\label{D4-split}
&\left\|\p_{x}^{l}D_{4}(\cdot, t)\right\|_{L^{\infty}}
\le C\sum_{j=0}^{l+1}(1+t)^{-\frac{l+1-j}{2}}\int_{0}^{t}\left\|\p_{x}^{j}J[(1-\p_{x}^{2})^{-1}\p_{x}^{2}(u^{2})](\cdot, t, \tau)\right\|_{L^{\infty}}d\tau\\
&\le C\sum_{j=0}^{l+1}(1+t)^{-\frac{l+1-j}{2}}\left(\int_{0}^{t/2}+\int_{t/2}^{t}\right)\left\|\p_{x}^{j}J[(1-\p_{x}^{2})^{-1}\p_{x}^{2}(u^{2})](\cdot, t, \tau)\right\|_{L^{\infty}}d\tau\\
&=:C\sum_{j=0}^{l+1}(1+t)^{-\frac{l+1-j}{2}}(D_{4.1}+D_{4.2}).
\end{split}
\end{align}
Now, let us treat $D_{4.1}$. First, we note that 
\[(1-\p_{x}^{2})^{-1}\p_{x}^{2}(u^{2})=(1-\p_{x}^{2})^{-1}\p_{x}(2uu_{x})=2(1-\p_{x}^{2})^{-1}\left((u_{x})^{2}+uu_{xx}\right).\]
Therefore, by using \eqref{J}, \eqref{eta-est2}, Young's inequality, the above equation, Lemma~\ref{lem.embedding}, Schwarz's inequality, \eqref{heat-decay} and \eqref{u-est-L2}, we can see that 
\begin{align}
\begin{split}\label{D4.1-est}
D_{4.1}
&\le C\int_{0}^{t/2}\left\|\p_{x}^{j}G(\cdot, t-\tau)\right\|_{L^{\infty}}\left\|(1-\p_{x}^{2})^{-1}\p_{x}^{2}(u^{2})(\cdot, \tau)\right\|_{L^{1}}d\tau \\
&\le C\int_{0}^{t/2}\left\|\p_{x}^{j}G(\cdot, t-\tau)\right\|_{L^{\infty}}\left(\left\|(u_{x})^{2}(\cdot, \tau)\right\|_{L^{1}}+\|(uu_{xx})(\cdot, \tau)\|_{L^{1}}\right)d\tau \\
&\le C\int_{0}^{t/2}\left\|\p_{x}^{j}G(\cdot, t-\tau)\right\|_{L^{\infty}}\left(\|u_{x}(\cdot, \tau)\|_{L^{2}}^{2}+\|u(\cdot, \tau)\|_{L^{2}}\|u_{xx}(\cdot, \tau)\|_{L^{2}}\right)d\tau \\
&\le CE_{s}\int_{0}^{t/2}(t-\tau)^{-\frac{1}{2}-\frac{j}{2}}\left\{(1+\tau)^{-\frac{3}{2}}+(1+\tau)^{-\frac{1}{4}}(1+\tau)^{-\frac{5}{4}}\right\}d\tau  \\
&\le CE_{s}(1+t)^{-\frac{1}{2}-\frac{j}{2}}, \ \ t\ge1.
\end{split}
\end{align}
To evaluate $D_{4.2}$, making the integration by parts, we rewrite $J[(1-\p_{x}^{2})^{-1}\p_{x}^{2}(u^{2})](x, t, \tau)$ as follows:
\begin{align}
\begin{split}\label{int-parts-D4}
J[(1-\p_{x}^{2})^{-1}\p_{x}^{2}(u^{2})](x, t, \tau)&
=\int_{\R}G(x-y, t-\tau)\eta(y, \tau)^{-1}(1-\p_{y}^{2})^{-1}\p_{y}^{2}(u^{2})(y, \tau)dy\\
&=\int_{\R}\p_{x}G(x-y, t-\tau)\eta(y, \tau)^{-1}(1-\p_{y}^{2})^{-1}\p_{y}(u^{2})(y, \tau)dy\\
&\ \ \ \ -\int_{\R}G(x-y, t-\tau)\p_{y}(\eta(y, \tau)^{-1})(1-\p_{y}^{2})^{-1}\p_{y}(u^{2})(y, \tau)dy.
\end{split}
\end{align}
Similarly as \eqref{prepare-D2.2}, for $i=0, 1$, we have from \eqref{eta-est2}, Lemma~\ref{lem.eta-decay} and Lemma~\ref{lem.embedding} that 
\begin{equation}\label{prepare-D4.2-1}
\left\|\p_{x}^{j}\left(\p_{x}^{i}(\eta(\cdot, \tau)^{-1})(1-\p_{x}^{2})^{-1}\p_{x}(u^{2})(\cdot, \tau)\right)\right\|_{L^{2}}
\le C\sum_{n=0}^{j}(1+\tau)^{-\frac{j+i-n}{2}}\left\|\p_{x}^{n+1}(u^{2})(\cdot, \tau)\right\|_{L^{2}}.
\end{equation}
Here, from \eqref{u-est-L2} and \eqref{u-est-Linfinity}, we obtain the following estimate for $n=0, 1, \cdots, j$.
\begin{align}\label{prepare-D4.2-2}
\begin{split}
&\left\|\p_{x}^{n+1}(u^{2})(\cdot, \tau)\right\|_{L^{2}}
\le C\sum_{m=0}^{n+1}\left\|\p_{x}^{n+1-m}u(\cdot, \tau)\p_{x}^{m}u(\cdot, \tau)\right\|_{L^{2}} \\
&\le C\left(\|u(\cdot, \tau)\|_{L^{\infty}}\left\|\p_{x}^{n+1}u(\cdot, \tau)\right\|_{L^{2}}+\sum_{m=0}^{n}\left\|\p_{x}^{n+1-m}u(\cdot, \tau)\right\|_{L^{2}}\left\|\p_{x}^{m}u(\cdot, \tau)\right\|_{L^{\infty}}\right) \\
&\le CE_{s}\left\{(1+\tau)^{-\frac{1}{2}}(1+\tau)^{-\frac{3}{4}-\frac{n}{2}}+\sum_{m=0}^{n}(1+\tau)^{-\frac{3}{4}-\frac{n-m}{2}}(1+\tau)^{-\frac{1}{2}-\frac{m}{2}}\right\} \\
&\le CE_{s}(1+\tau)^{-\frac{5}{4}-\frac{n}{2}}, \ \ t\ge0.
\end{split}
\end{align}
Therefore, combining \eqref{prepare-D4.2-1} and \eqref{prepare-D4.2-2}, for $i=0, 1$, we get 
\begin{equation}\label{prepare-D4.2-3}
\left\|\p_{x}^{j}\left(\p_{x}^{i}(\eta(\cdot, \tau)^{-1})(1-\p_{x}^{2})^{-1}\p_{x}(u^{2})(\cdot, \tau)\right)\right\|_{L^{2}}
\le CE_{s}(1+\tau)^{-\frac{5}{4}-\frac{j+i}{2}}.
\end{equation}
Thus, from \eqref{int-parts-D4}, Young's inequality, \eqref{heat-decay} and \eqref{prepare-D4.2-3}, we have
\begin{align}
\begin{split}\label{D4.2-est}
D_{4.2}&\le C\int_{t/2}^{t}\left\|\p_{x}G(\cdot, t-\tau)\right\|_{L^{2}}\left\|\p_{x}^{j}\left(\eta(\cdot, \tau)^{-1}(1-\p_{x}^{2})^{-1}\p_{x}(u^{2})(\cdot, \tau)\right)\right\|_{L^{2}}d\tau \\
&\ \ \ +C\int_{t/2}^{t}\|G(\cdot, t-\tau)\|_{L^{2}}\left\|\p_{x}^{j}\left(\p_{x}(\eta(\cdot, \tau)^{-1})(1-\p_{x}^{2})^{-1}\p_{x}(u^{2})(\cdot, \tau)\right)\right\|_{L^{2}}d\tau \\
&\le CE_{s}\int_{t/2}^{t}(t-\tau)^{-\frac{3}{4}}(1+\tau)^{-\frac{5}{4}-\frac{j}{2}}d\tau
+CE_{s}\int_{t/2}^{t}(t-\tau)^{-\frac{1}{4}}(1+\tau)^{-\frac{7}{4}-\frac{j}{2}}d\tau\\
&\le CE_{s}(1+t)^{-1-\frac{j}{2}}, \ \ t\ge0.
\end{split}
\end{align}
Finally, combining \eqref{D4-split}, \eqref{D4.1-est} and \eqref{D4.2-est}, we obtain 
\begin{equation}\label{D4-est}
\left\|\p_{x}^{l}D_{4}(\cdot, t)\right\|_{L^{\infty}}\le CE_{s}(1+t)^{-1-\frac{l}{2}}, \ \ t\ge1, \ \ 0\le l\le s-2.
\end{equation}

Summarizing up \eqref{integral-eq-D}, \eqref{D1-est}, \eqref{D2-est}, \eqref{D3-est} and \eqref{D4-est}, we eventually arrive at 
\begin{align*}
\left\|\p_{x}^{l}D[u, \psi](\cdot, t)\right\|_{L^{\infty}}
&\le CE_{s}(1+t)^{-1-\frac{l}{2}}\log(2+t)+CE_{s}(1+t)^{-1-\frac{l}{2}} \\
&\ \ \ +C\begin{cases}
(1+t)^{-\alpha+\frac{1}{2}-\frac{l}{2}}, &1<\alpha<3/2, \\
(1+t)^{-1-\frac{l}{2}}\log(2+t), &\alpha=3/2, \\
(1+t)^{-1-\frac{l}{2}},  &3/2<\alpha<2
\end{cases} \\
&\le C\begin{cases}
(1+t)^{-\alpha+\frac{1}{2}-\frac{l}{2}}, &t\ge1, \ 1<\alpha<3/2, \\
(1+t)^{-1-\frac{l}{2}}\log(2+t),  &t\ge1, \ 3/2\le \alpha<2,
\end{cases}\quad 0\le l \le s-2.
\end{align*}
This completes the proof.
\end{proof}

\begin{proof}[\rm{\bf{End of the Proof of Theorem~\ref{main.thm-2nd.AP}~for~$\bm{1<\alpha<2}$}}]
It follows from \eqref{difference} and \eqref{integral-eq-2nd} that 
\begin{equation*}
u(x, t)-\chi(x, t)-Z(x, t)=U[\psi_{0}](x, t, 0)-Z(x, t)+D[u, \psi](x, t), 
\end{equation*}
where $Z(x, t)$, $U$ and $D[u, \psi](x, t)$ are defined by \eqref{Second-AP-Z}, \eqref{U} and \eqref{integral-eq-D}, respectively. Then, by using Proposition~\ref{prop.D-est} to the above equation, we obtain the following estimate: 
\begin{align*}
&\left\|\p_{x}^{l}(u(\cdot, t)-\chi(\cdot, t)-Z(\cdot, t))\right\|_{L^{\infty}} \\
&\le \left\|\p_{x}^{l}(U[\psi_{0}](\cdot, t, 0)-Z(\cdot, t))\right\|_{L^{\infty}}
+C\begin{cases}
(1+t)^{-\alpha+\frac{1}{2}-\frac{l}{2}}, &t\ge1, \ 1<\alpha<3/2, \\
(1+t)^{-1-\frac{l}{2}}\log(2+t), &t\ge1, \ 3/2\le \alpha<2
\end{cases}
\end{align*}
for all integer $l$ satisfying $0\le l \le s-2$. 
Therefore, from \eqref{linear-2nd.AP-1}, we finally obtain 
\[\limsup_{t\to \infty}(1+t)^{\frac{\alpha}{2}+\frac{l}{2}}\left\|\p_{x}^{l}(u(\cdot, t)-\chi(\cdot, t)-Z(\cdot, t))\right\|_{L^{\infty}}=0.\]
Thus, we get \eqref{main.thm-2nd.AP-1}. It means that the proof of Theorem~\ref{main.thm-2nd.AP}~for~$1<\alpha<2$ has completed. 
\end{proof}

\subsection{Proof of Theorem \ref{main.thm-2nd.AP}~for~$\bm{\alpha\ge2}$}  

Finally in this last subsection, we shall complete the proof of Theorem~\ref{main.thm-2nd.AP} in the case of $\alpha\ge2$. Namely, we prove \eqref{main.thm-2nd.AP-2} and \eqref{main.thm-2nd.AP-3}.  First, let us introduce the following auxiliary Cauchy problem: 
\begin{align}\label{second-aux}
\begin{split}
v_{t}+(\beta \chi v)_{x}-v_{xx}&=-\gamma \chi_{xxx}, \ \ x\in \R, \ t>0, \\
v(x, 0)&=0, \ \ x\in \R. 
\end{split}
\end{align}
Then, the leading term of the solution $v(x, t)$ to \eqref{second-aux} is given by $V(x, t)$ defined by \eqref{Second-AP-V}. Actually, the following asymptotic formula is true (for the proof, see Proposition~4.3 in \cite{F19-1}). 
\begin{prop}\label{lem.second-aux}
Let $l$ be a non-negative integer. If $|M| \le 1$, then we have
\begin{equation}\label{second-formula}
\left\|\p_{x}^{l}(v(\cdot, t)-V(\cdot, t))\right\|_{L^{\infty}} \le C|M|(1+t)^{-1-\frac{l}{2}}, \ \ t\ge1. 
\end{equation}
Here, $v(x, t)$ is the solution to \eqref{second-aux} and $V(x, t)$ is defined by \eqref{Second-AP-V}.
\end{prop}
\noindent 
This asymptotic formula plays an important role to complete the proof of \eqref{main.thm-2nd.AP-2} and \eqref{main.thm-2nd.AP-3}. 

Next, we shall further transform \eqref{BBMB-up} to the following appropriate expression:
\begin{align}
\label{BBMB-up-2}
\begin{split}
&u_{t}-u_{xxt}-u_{xx}+\gamma u_{xxx}+\beta uu_{x}=0\\
\Longleftrightarrow \ \ &u_{t}+\beta uu_{x}-u_{xx}=-\gamma (1-\p_{x}^{2})^{-1}\p_{x}^{3}u+(1-\p_{x}^{2})^{-1}\p_{x}^{4}u-\frac{\beta}{2} (1-\p_{x}^{2})^{-1}\p_{x}^{3}(u^{2}) \\
\Longleftrightarrow \ \ &u_{t}+\beta uu_{x}+\gamma u_{xxx}-u_{xx}=(1-\p_{x}^{2})^{-1}\p_{x}^{4}u-\frac{\beta}{2} (1-\p_{x}^{2})^{-1}\p_{x}^{3}(u^{2})-\gamma (1-\p_{x}^{2})^{-1}\p_{x}^{5}u. 
\end{split}
\end{align}
Now, we set 
\begin{equation}\label{difference-up-2}
w(x, t):=u(x, t)-\chi(x, t)-v(x, t)=\psi(x, t)-v(x, t). 
\end{equation}
Then, from \eqref{BBMB-up-2}, \eqref{Burgers} and \eqref{second-aux}, we have the following Cauchy problem: 
\begin{align}\label{difference-second-2}
\begin{split}
w_{t}+(\beta \chi w)_{x}-w_{xx}
&=-\p_{x}\left(\frac{\beta}{2}\psi^{2}\right)
-\gamma \psi_{xxx}
+(1-\p_{x}^{2})^{-1}\p_{x}^{4}u\\
&\ \ \ \,-\frac{\beta}{2}(1-\p_{x}^{2})^{-1}\p_{x}^{3}(u^{2})
-\gamma(1-\p_{x}^{2})^{-1}\p_{x}^{5}u, \ \ x\in \R, \ t>0,\\
w(x, 0)&=u_{0}(x)-\chi_{*}(x)=\psi_{0}(x), \ \ x\in \R.
\end{split}
\end{align}
In the same way to get \eqref{integral-eq-D}, applying Lemma~\ref{lem.formula} to \eqref{difference-second-2}, we obtain
\begin{equation}\label{integral-eq-2nd-2}
w(x, t)=U[\psi_{0}](x, t, 0)+N[u, \psi](x, t), 
\end{equation}
where the integral part $N[u, \psi](x, t)$ is defined by
\begin{align}\label{integral-eq-N}
\begin{split}
&N[u, \psi](x, t):=-\frac{\beta}{2}\int_{0}^{t}U[\p_{x}(\psi^{2})(\tau)](x, t, \tau)d\tau
-\gamma \int_{0}^{t}U[(1-\p_{x}^{2})^{-1}\p_{x}^{3}\psi(\tau)](x, t, \tau)d\tau \\
&\quad +\int_{0}^{t}U[(1-\p_{x}^{2})^{-1}\p_{x}^{4}u(\tau)](x, t, \tau)d\tau 
-\frac{\beta}{2}\int_{0}^{t}U[(1-\p_{x}^{2})^{-1}\p_{x}^{3}(u^{2})(\tau)](x, t, \tau)d\tau \\
&\quad -\gamma \int_{0}^{t}U[(1-\p_{x}^{2})^{-1}\p_{x}^{5}u(\tau)](x, t, \tau)d\tau \\
&\quad =:N_{1}(x, t)+N_{2}(x, t)+N_{3}(x, t)+N_{4}(x, t)+N_{5}(x, t). 
\end{split}
\end{align}

To complete the proof of \eqref{main.thm-2nd.AP-2} and \eqref{main.thm-2nd.AP-3}, it is sufficient to lead the following proposition: 
\begin{prop}\label{prop.N-est}
Let $s\ge2$ be an integer and $\alpha\ge2$. Assume that the initial data $u_{0}(x)$ satisfies the condition \eqref{data}, $u_{0}\in H^{s}(\R)$ and $\|u_{0}\|_{H^{s}}+\|u_{0}\|_{L^{1}}$ is sufficiently small. Then, the estimate  
\begin{equation}\label{N-est}
\left\|\p_{x}^{l}N[u, \psi](\cdot, t)\right\|_{L^{\infty}}
\le C(1+t)^{-1-\frac{l}{2}}, \ \ t\ge1
\end{equation}
holds for any integer $l$ satisfying $0\le l\le s-2$.
\end{prop}
\begin{proof}
We shall evaluate each term $N_{i}(x, t)$ ($i=1, 2, 3, 4, 5$) in the right hand side of \eqref{integral-eq-N}. 
First, we start with evaluation of $N_{1}(x, t)$. From the definitions \eqref{integral-eq-D} and \eqref{integral-eq-N}, we note that $N_{1}(x, t)\equiv D_{1}(x, t)$. However, we need to re-evaluate it for the case of $\alpha\ge2$. Now, recalling \eqref{D1-split} and replacing $D_{1}(x, t)$ with $N_{1}(x, t)$, we have
\begin{align}
\begin{split}\label{N1-split}
\left\|\p_{x}^{l}N_{1}(\cdot, t)\right\|_{L^{\infty}}
&\le C\sum_{j=0}^{l+1}(1+t)^{-\frac{l+1-j}{2}}\left(\int_{0}^{t/2}+\int_{t/2}^{t}\right)\left\|\p_{x}^{j}J[\psi^{2}](\cdot, t, \tau)\right\|_{L^{\infty}}d\tau\\
&=:C\sum_{j=0}^{l+1}(1+t)^{-\frac{l+1-j}{2}}(N_{1.1}+N_{1.2}).
\end{split}
\end{align}
Similarly as $D_{1.1}$, from \eqref{J}, Young's inequality, \eqref{eta-est2}, \eqref{heat-decay} and \eqref{main.thm-1st.AP-L2}, we get
\begin{align}
\begin{split}\label{N1.1-est}
N_{1.1}&\le C\int_{0}^{t/2}\left\|\p_{x}^{j}G(\cdot, t-\tau)\right\|_{L^{\infty}}\left\|\psi^{2}(\cdot, \tau)\right\|_{L^{1}}d\tau \\
&\le C\int_{0}^{t/2}(t-\tau)^{-\frac{1}{2}-\frac{j}{2}}(1+\tau)^{-\frac{3}{2}+2\e}d\tau 
\le C(1+t)^{-\frac{1}{2}-\frac{j}{2}}, \ \ t\ge1.
\end{split}
\end{align}
For $N_{1.2}$, modifying \eqref{D1.2-est-2} a little bit, we get the following estimate for $N_{1.2}$:
\begin{align}
\begin{split}\label{N1.2-est}
N_{1.2}&\le C\int_{t/2}^{t}\|G(\cdot, t-\tau)\|_{L^{1}}\left\|\p_{x}^{j}((\eta^{-1}\psi^{2})(\cdot, \tau))\right\|_{L^{\infty}}d\tau\\
&\le C\int_{t/2}^{t}(1+\tau)^{-2-\frac{j}{2}+2\e}d\tau \le C(1+t)^{-1-\frac{j}{2}+2\e}, \ \ t\ge0, 
\end{split}
\end{align}
where we have used the following estimate corresponding to \eqref{prepare-for-D1}:
\[
\left\|\p_{x}^{j}((\eta^{-1}\psi^{2})(\cdot, \tau))\right\|_{L^{\infty}}
\le C(1+\tau)^{-2-\frac{j}{2}+2\e}, \ \ t\ge0.
\]
Therefore, summing up \eqref{N1-split}, \eqref{N1.1-est} and \eqref{N1.2-est}, we obtain  
\begin{align}\label{N1-est}
\begin{split}
\left\|\p_{x}^{l}N_{1}(\cdot, t)\right\|_{L^{\infty}}
&\le C(1+t)^{-1-\frac{l}{2}}+C(1+t)^{-\frac{3}{2}-\frac{l}{2}+2\e} \\
&\le C(1+t)^{-1-\frac{l}{2}}, \ \ t\ge1, \ 0\le l \le s-2.
\end{split}
\end{align}

Next, we deal with the estimate for $N_{2}(x, t)$. Similarly as \eqref{D2-split}, we have 
\begin{align}
\begin{split}\label{N2-split}
\left\|\p_{x}^{l}N_{2}(\cdot, t)\right\|_{L^{\infty}}
&\le C\sum_{j=0}^{l+1}(1+t)^{-\frac{l+1-j}{2}}\left(\int_{0}^{t/2}+\int_{t/2}^{t}\right)\left\|\p_{x}^{j}J[(1-\p_{x}^{2})^{-1}\p_{x}^{2}\psi](\cdot, t, \tau)\right\|_{L^{\infty}}d\tau\\
&=:C\sum_{j=0}^{l+1}(1+t)^{-\frac{l+1-j}{2}}(N_{2.1}+N_{2.2}).
\end{split}
\end{align}
Also, making the integration by parts, we rewrite $J[(1-\p_{x}^{2})^{-1}\p_{x}^{2}\psi](x, t, \tau)$ as follows:  
\begin{align}
\begin{split}\label{int-parts-N}
J[(1-\p_{x}^{2})^{-1}\p_{x}^{2}\psi](x, t, \tau)&
=\int_{\R}G(x-y, t-\tau)\eta(y, \tau)^{-1}(1-\p_{y}^{2})^{-1}\p_{y}^{2}\psi(y, \tau)dy\\
&=\int_{\R}\p_{x}G(x-y, t-\tau)\eta(y, \tau)^{-1}(1-\p_{y}^{2})^{-1}\p_{y}\psi(y, \tau)dy\\
&\ \ \ \ -\int_{\R}G(x-y, t-\tau)\p_{y}(\eta(y, \tau)^{-1})(1-\p_{y}^{2})^{-1}\p_{y}\psi(y, \tau)dy.
\end{split}
\end{align}
Therefore, in the same way to get \eqref{D2.1-est}, by using \eqref{int-parts}, Young's inequality, \eqref{eta-est2}, \eqref{heat-decay}, Lemma~\ref{lem.embedding}, Schwarz's
inequality, \eqref{main.thm-1st.AP-L2} and Lemma~\ref{lem.eta-decay}, we can see that 
\begin{align}
\begin{split}\label{N2.1-est}
&N_{2.1}\le  C\int_{0}^{t/2}\left\|\p_{x}^{j+1}G(\cdot, t-\tau)\right\|_{L^{2}}\left\|\p_{x}\psi(\cdot, \tau)\right\|_{L^{2}}d\tau \\
&\qquad \quad+C\int_{0}^{t/2}\left\|\p_{x}^{j}G(\cdot, t-\tau)\right\|_{L^{\infty}}\left\|\p_{x}(\eta(\cdot, \tau)^{-1})\right\|_{L^{2}}\left\|\p_{x}\psi(\cdot, \tau)\right\|_{L^{2}}d\tau \\
&\le CE_{s}\int_{0}^{t/2}(t-\tau)^{-\frac{3}{4}-\frac{j}{2}}(1+\tau)^{-\frac{5}{4}+\e}d\tau
+CE_{s}\int_{0}^{t/2}(t-\tau)^{-\frac{1}{2}-\frac{j}{2}}(1+\tau)^{-\frac{1}{4}}(1+\tau)^{-\frac{5}{4}+\e}d\tau\\
&\le CE_{s}(1+t)^{-\frac{3}{4}-\frac{j}{2}}+CE_{s}(1+t)^{-\frac{1}{2}-\frac{j}{2}}
\le CE_{s}(1+t)^{-\frac{1}{2}-\frac{j}{2}}, \ \ t\ge1.
\end{split}
\end{align}
For $i=0, 1$, in the same way as \eqref{prepare-D2.2}, by using \eqref{eta-est2}, Lemma~\ref{lem.eta-decay}, Lemma~\ref{lem.embedding} and \eqref{main.thm-1st.AP-L2}, we get 
\begin{equation}\label{prepare-N2.2}
\left\|\p_{x}^{j}\left(\p_{x}^{i}(\eta(\cdot, \tau)^{-1})(1-\p_{x}^{2})^{-1}\p_{x}\psi(\cdot, \tau)\right)\right\|_{L^{2}}
\le CE_{s}(1+\tau)^{-\frac{5}{4}-\frac{i+j}{2}+\e}, \ \ t\ge0.
\end{equation}
Thus, from \eqref{int-parts-N}, Young's inequality, \eqref{heat-decay} and \eqref{prepare-N2.2}, we obtain 
\begin{align}
\begin{split}\label{N2.2-est}
N_{2.2}&\le C\int_{t/2}^{t}\left\|\p_{x}G(\cdot, t-\tau)\right\|_{L^{2}}\left\|\p_{x}^{j}\left(\eta(\cdot, \tau)^{-1}(1-\p_{x}^{2})^{-1}\p_{x}\psi(\cdot, \tau)\right)\right\|_{L^{2}}d\tau \\
&\ \ \ +C\int_{t/2}^{t}\|G(\cdot, t-\tau)\|_{L^{2}}\left\|\p_{x}^{j}\left(\p_{x}(\eta(\cdot, \tau)^{-1})(1-\p_{x}^{2})^{-1}\p_{x}\psi(\cdot, \tau)\right)\right\|_{L^{2}}d\tau \\
&\le CE_{s}\int_{t/2}^{t}(t-\tau)^{-\frac{3}{4}}(1+\tau)^{-\frac{5}{4}-\frac{j}{2}+\e}d\tau
+CE_{s}\int_{t/2}^{t}(t-\tau)^{-\frac{1}{4}}(1+\tau)^{-\frac{7}{4}-\frac{j}{2}+\e}d\tau\\
&\le CE_{s}(1+t)^{-1-\frac{j}{2}+\e}, \ \ t\ge0.
\end{split}
\end{align}
Combining \eqref{N2-split}, \eqref{N2.1-est} and \eqref{N2.2-est}, we obtain 
\begin{align}\label{N2-est}
\begin{split}
\left\|\p_{x}^{l}N_{2}(\cdot, t)\right\|_{L^{\infty}}
&\le C(1+t)^{-1-\frac{l}{2}}+C(1+t)^{-\frac{3}{2}-\frac{l}{2}+\e} \\
&\le C(1+t)^{-1-\frac{l}{2}}, \ \ t\ge1, \ 0\le l \le s-2.
\end{split}
\end{align}

Next, we shall treat $N_{3}(x, t)$ and $N_{4}(x, t)$. For these terms, we can use \eqref{D3-est} and \eqref{D4-est} as the estimates for $N_{3}(x, t)$ and $N_{4}(x, t)$, respectively. Actually, from the definitions of $D[u, \psi](x, t)$ and $N[u, \psi](x, t)$, i.e. \eqref{integral-eq-D} and \eqref{integral-eq-N}, we see that $N_{3}(x, t)\equiv D_{3}(x, t)$ and $N_{4}(x, t)\equiv D_{4}(x, t)$. In addition, the estimates \eqref{D3-est} and \eqref{D4-est} does not depend on $\alpha$. Therefore, we immediately get 
\begin{equation}\label{N3-N4-est}
\left\|\p_{x}^{l}N_{i}(\cdot, t)\right\|_{L^{\infty}}\le CE_{s}(1+t)^{-1-\frac{l}{2}}, \ \ t\ge1, \ 0\le l \le s-2, \ i=3, 4.
\end{equation}

Finally, let us explain about the evaluation for $N_{5}(x, t)$. From Lemma~\ref{lem.est-U2}, we can see that 
\begin{align}\label{N5-def}
\begin{split}
\left\|\p_{x}^{l}N_{5}(\cdot, t)\right\|_{L^{\infty}}
\le C\sum_{j=0}^{l+1}(1+t)^{-\frac{1}{2}(l+1-j)}\int_{0}^{t}\left\|\p_{x}^{j}J[(1-\p_{x}^{2})^{-1}\p_{x}^{4}u](\cdot, t, \tau)\right\|_{L^{\infty}}d\tau.
\end{split}
\end{align}
Therefore, from the structure of the right hand side of \eqref{N5-def}, it can be realized that to evaluate $N_{5}(x, t)$, we need to get the estimate for $(1-\p_{x}^{2})^{-1}\p_{x}^{4}u(x, t)$. Now, by using Plancherel's theorem and \eqref{nonlocal-op}, instead of \eqref{embedding-2}, we obtain 
\begin{align}\label{embedding-3}
\begin{split}
\left\|(1-\p_{x}^{2})^{-1}\p_{x}^{j+2}u(\cdot, t)\right\|_{L^{2}}
&=\left\|\frac{(i\xi)^{j+2}}{1+\xi^{2}}\widehat{u}(\xi, t)\right\|_{L^{2}_{\xi}}
=\left\|\frac{\xi^{2}}{1+\xi^{2}}(i\xi)^{j}\widehat{u}(\xi, t)\right\|_{L^{2}_{\xi}}  \\
&\le \left\|(i\xi)^{j}\widehat{u}(\xi, t)\right\|_{L^{2}_{\xi}}
=\left\|\p_{x}^{j}u(\cdot, t)\right\|_{L^{2}}. 
\end{split}
\end{align}
By virtue of \eqref{embedding-3}, we can evaluate $N_{5}(x, t)$ in essentially the same way to get \eqref{D3-est}. Indeed, it follows from \eqref{N5-def} that the following estimate has established (we omit the details):   
\begin{equation}\label{N5-est}
\left\|\p_{x}^{l}N_{5}(\cdot, t)\right\|_{L^{\infty}}\le CE_{s}(1+t)^{-1-\frac{l}{2}}, \ \ t\ge1, \ \ 0\le l\le s-2.
\end{equation}

Summarizing up \eqref{integral-eq-N}, \eqref{N1-est}, \eqref{N2-est}, \eqref{N3-N4-est} and \eqref{N5-est}, we eventually arrive at 
\[
\left\|\p_{x}^{l}N[u, \psi](\cdot, t)\right\|_{L^{\infty}}
\le C(1+t)^{-1-\frac{l}{2}}, \ \ t\ge1, \ 0\le l\le s-2.
\]
This completes the proof.
\end{proof}

\begin{proof}[\rm{\bf{End of the Proof of Theorem~\ref{main.thm-2nd.AP}~for~$\bm{\alpha\ge2}$}}]
It follows from \eqref{difference-up-2} and \eqref{integral-eq-2nd-2} that 
\begin{equation}\label{u-chi-V}
u(x, t)-\chi(x, t)-V(x, t)=U[\psi_{0}](x, t, 0)+v(x, t)-V(x, t)+N[u, \psi](x, t), 
\end{equation}
where $v(x, t)$ is the solution to \eqref{second-aux}, and $U$, $V(x, t)$ and $N[u, \psi](x, t)$ are defined by \eqref{U}, \eqref{Second-AP-V} and \eqref{integral-eq-N}, respectively. In what follows, the proof is divided into two parts below:

\medskip
\noindent
\underline{{\rm (i)} For $\alpha>2$}: Applying Lemma~\ref{lem.est-U1}, Propositions~\ref{lem.second-aux} and \ref{prop.N-est} to \eqref{u-chi-V}, we immediately obtain 
\[
\left\|\p_{x}^{l}(u(\cdot, t)-\chi(\cdot, t)-V(\cdot, t))\right\|_{L^{\infty}}
\le C(1+t)^{-1-\frac{l}{2}}, \ \ t\ge1 
\]
for all integer $l$ satisfying $0\le l \le s-2$. This means that \eqref{main.thm-2nd.AP-3} has established. 

\medskip
\noindent
\underline{{\rm (ii)} For $\alpha=2$}: By subtracting $Z(x, t)$ (being defined by \eqref{Second-AP-Z}) from both sides of \eqref{u-chi-V}, we get
\begin{equation*}
u(x, t)-\chi(x, t)-Z(x, t)-V(x, t)=U[\psi_{0}](x, t, 0)-Z(x, t)+v(x, t)-V(x, t)+N[u, \psi](x, t). 
\end{equation*}
Then, by using Propositions~\ref{lem.second-aux} and \ref{prop.N-est} to the above equation, we obtain the following estimate: 
\begin{equation*}
\left\|\p_{x}^{l}(u(\cdot, t)-\chi(\cdot, t)-Z(\cdot, t)-V(\cdot, t))\right\|_{L^{\infty}} 
\le \left\|\p_{x}^{l}(U[\psi_{0}](\cdot, t, 0)-Z(\cdot, t))\right\|_{L^{\infty}}
+C(1+t)^{-1-\frac{l}{2}}, \ \ t\ge1
\end{equation*}
for all integer $l$ satisfying $0\le l \le s-2$. 
Therefore, from \eqref{linear-2nd.AP-2}, we eventually arrive at  
\[\limsup_{t\to \infty}\frac{(1+t)^{1+\frac{l}{2}}}{\log(1+t)}\left\|\p_{x}^{l}(u(\cdot, t)-\chi(\cdot, t)-Z(\cdot, t)-V(\cdot, t))\right\|_{L^{\infty}}=0.\]
Thus, we are able to see that \eqref{main.thm-2nd.AP-2} is true. 

\smallskip
As a conclusion, the asymptotic relation \eqref{main.thm-2nd.AP-2} and the inequality \eqref{main.thm-2nd.AP-3} have been proven. It means that the proof of Theorem~\ref{main.thm-2nd.AP}~for~$\alpha\ge 2$ has completed. 
\end{proof}

\section*{Acknowledgments}
The first author is supported by Grant-in-Aid for Research Activity Start-up No.20K22303, Japan Society for the Promotion of Science. The second author is supported by JST CREST Grant Number JPMJCR1913, Japan and Grant-in-Aid for Young Scientists Research No.19K14581, Japan Society for the Promotion of Science.



\bigskip
\par\noindent
\begin{flushleft}Ikki Fukuda\\
Division of Mathematics and Physics, \\
Faculty of Engineering, \\
Shinshu University\\
4-17-1, Wakasato, Nagano, 380-8553, JAPAN\\
E-mail: i\_fukuda@shinshu-u.ac.jp
\vskip15pt
Masahiro Ikeda\\
Department of Mathematics, \\
Faculty of Science and Technology, \\
Keio University \\
3-14-1, Hiyoshi, Kohoku-ku, Yokohama, 223-8522, JAPAN/\\
Center for Advanced Intelligence Project,\\ RIKEN, \\
Nihonbashi 1-chome Mitsui Building, 15th floor, 1-4-1 Nihonbashi,Chuo-ku, Tokyo, 103-0027, JAPAN \\
E-mail: masahiro.ikeda@keio.jp/masahiro.ikeda@riken.jp
\end{flushleft}

\end{document}